\documentclass[letterpaper, 10pt, twocolumn]{ieeetran}

\IEEEoverridecommandlockouts                              

\usepackage{amsmath,amssymb}
\usepackage{euscript,yfonts,psfrag,latexsym,dsfont,graphicx,bbm,color}
\usepackage{amstext,wasysym,pdfsync}
\usepackage{epstopdf, bm}
\usepackage{amsmath, amssymb, psfrag, bm, latexsym, color, amstext}
\usepackage{graphicx}
\usepackage{epstopdf}
\usepackage{mathtools}
\usepackage{amsthm}
\usepackage{verbatim}
\usepackage{subcaption}
\usepackage[normalem]{ulem}
\usepackage{enumerate}

\usepackage{url}
\usepackage{float}

\usepackage{algorithm}
\usepackage{algorithmic}

\usepackage[normalem]{ulem}

\DeclareMathOperator*{\maxmin}{\max/\min}
\newcommand{\mR}{{\mathbb R}}
\newcommand{\mB}{{\mathbb B}}

\newcommand{\Belement}{b}
\newcommand{\Bbound}{\beta}
\newcommand{\Ccal}{\mathcal{C}_+}
\newcommand{\ordo}{\mathcal{O}}

\newtheorem{thm}{Theorem}
\newtheorem{cor}[thm]{Corollary}
\newtheorem{lemma}[thm]{Lemma}
\newtheorem{prop}[thm]{Proposition}

\newtheorem{remark}{Remark}

\newtheorem{assumption}{Assumption}

\newcommand{\Ltwo}[1]{L_{2}(#1)}
\newcommand{\Linf}[1]{L_{\infty}(#1)}

\graphicspath{{figures/}{../}}
\usepackage{xcolor}
\usepackage{etoolbox}

\title{\LARGE \bf
Modeling collective behaviors: A moment-based approach
}

\author{Silun Zhang$^{1}$, Axel Ringh$^{2}$, Xiaoming Hu$^{3}$, and Johan Karlsson$^{3}$
\thanks{*This work was supported by the Swedish Research Council (VR), 
 and by the ACCESS Linnaeus Center, KTH Royal Institute of Technology. }
\thanks{$^{1}$ MIT Laboratory for Information and Decision Systems (LIDS), Massachusetts Institute of Technology, Cambridge, USA. {\tt\small silunz@mit.edu}.}
\thanks{$^{2}$Department of Electronic and Computer Engineering, The Hong Kong University of Science and Technology, Hong Kong, China. {\tt\small eeringh@ust.hk}.}
\thanks{$^{3}$Division of Optimization and Systems Theory, Department of Mathematics, KTH Royal Institute of Technology, Stockholm, Sweden.
{\tt\small hu@kth.se}, {\tt\small johan.karlsson@math.kth.se.}}%
}

\begin{document}

\maketitle
\thispagestyle{empty}
\pagestyle{empty}

\begin{abstract}

In this work we introduce an approach for modeling and analyzing collective behavior of a group of agents using moments. We represent the group of agents via their distribution and derive a method to estimate the dynamics of the moments. We use this to predict the evolution of the distribution of agents by first computing the moment trajectories and then use this to reconstruct the distribution of the agents. In the latter an inverse problem is solved in order to reconstruct a nominal distribution and to recover the macro-scale properties of the group of agents. The proposed method is applicable for several types of multi-agent systems, e.g., leader-follower systems. We derive error bounds for the moment trajectories and describe how to take these error bounds into account for computing the moment dynamics. The convergence of the moment dynamics is also analyzed for cases with monomial moments. To illustrate the theory, two numerical examples are given. In the first we consider a multi-agent system with interactions and compare the proposed method for several types of moments. In the second example we apply the framework to a leader-follower problem for modeling a pedestrian crowd.





\end{abstract}


\section{Introduction}


The study of collective behavior of crowds is important in numerous examples in both natural and social sciences, and in particular for understanding the macro-scale behavior of collectives based on micro-scale dynamics of each individual. This is essential in a wide range of applications, such as biology \cite{couzin2005effective}, material science \cite{puntes2004collective, singamaneni2011magnetic}, and macro-economy \cite{cont2000herd}. Also in social sciences, many collective phenomena in society can be treated in such a framework \cite{burger2014mean, piccoli2009pedestrian}, e.g., understanding movements of crowds and how to evacuate pedestrian crowds in panic situations \cite{helbing2009pedestrian}, \cite{qi2016design, wang2015modeling},
or how circulated opinions in social networks affect the public opinion in the wake of crisis incidents \cite{hegselmann2002opinion, yang2014opinion}.


Such systems typically consist of a large number of agents, often too large for modeling each agent individually. Moreover, in many cases the agents are exchangeable and distinguishing each agent may even not be desirable.
To model such system of
homogeneous agents
it is sufficient to consider
the distribution of the agents, which can be described by the nonnegative measure (occupation measure)
\begin{equation*}
d\mu(x) = \frac{1}{N}\sum_{i=1}^N \delta(x - x_i) dx,
\end{equation*}
where $x_i \in \mR^d$ is the state of agent $i$ and $\delta$ denotes the Dirac delta function.
To characterize the evolution of such distributions, one of the methods used is the mean-field theory, (see, e.g., \cite{ bolley2011stochastic, jabin2017mean}). In this approach, one lets the number of particles tend to infinity, in which case the distribution of the agents converges weakly to the solution of certain kinetic equations \cite{dobrushin1979vlasov}.
These kinetic equations are typically partial differential equations (PDEs), such as the Liouville equation \cite{brockett2012notes, liboff2003kinetic}, and the Jeans-Vlasov (or Vlasov) equation \cite{dobrushin1979vlasov, jabin2017mean}, in the case of agents governed by deterministic dynamics. If the individual dynamics is instead stochastic, as in, e.g., mean-field games \cite{huang2012social, huang2006large, lasry2007mean}, the resulting PDEs are the Fokker-Planck equation \cite{brockett2012notes} and the McKean-Vlasov equation \cite{jabin2017mean}.
However, solving such PDEs tends to be computationally expensive, and care needs to be taken in order to guarantee stability and consistency of the solutions.
In addition, suitable control design strategies in order to steer the overall behavior of such PDE systems are nontrivial.
In the deterministic case, which is also the focus of this work, problems such as state estimation, prediction, and observability of the distributions for linear multi-agent systems without interactions have been studied in, e.g., \cite{chen2017optimal, chen2018state, zeng2016ensemble}.

In this work, we propose a different approach to studying macro-scale behavior based on micro-scale agent models. In particular, we describe the agent distribution using moments
\[
m_k=\frac{1}{N}\sum_{i=1}^N \phi_k(x_i), \text{ for } k=1,\ldots, M,
\]
where $\phi_k$ are some kernel functions. With suitably selected kernel functions the moments convey the overall information of the distribution, which can be used to reconstruct nominal estimates and describe macroscopic properties of the distribution.
Moreover, the dynamics of the moments can be closely approximated by an ODE, which is obtained using the dynamics of the agents.
For computing the moment dynamics we utilize a lifting technique inspired by the Koopman operator framework \cite{budivsic2012applied}, \cite{ mauroy2016global, rowley2017model}.
Thereby, instead of directly addressing the nonlinear systems the problem is lifted into an infinite dimensional structured problem (cf. \cite{benamou2000computational, rowley2017model}) which
naturally admits approximations by finite dimensional linear or quadratic systems.

This moment based system representation gives rise to a model reduction technique for systems containing a large number of identical nonlinear subsystems and the reduced order model is obtained by solving a convex optimization problem.
Further, we derive error bounds on the resulting moment trajectories that are expressed in parameters that can be tuned in the optimization problems.
The theory is applicable for a wide range of applications, such as multi-agent systems with interactions as well as with leaders and/or control input, and the use of this framework can considerably reduce the computational burden for analyzing such systems.
In particular, compared to the mean-field method, the proposed method decouples the dependence of space and time,
which we will get back to in connection to the numerical examples in Section~\ref{sec:example}.
We therefore propose to develop and use this theory for multi-agent applications such as crowd dynamics, opinion dynamics and other macroscopic problems.

%
%
%

Partial results in this work have previously been reported in the conference paper \cite{zhang2018CDCmodeling}.
The outline of this article is as follows.  In Section~\ref{sec:background} we introduce background material on the moment problem and on the logarithmic norm. 
Section~\ref{sec:moment_rep} presents the main methodology for modeling multi-agent systems based on moments.
Section~\ref{sec:Convergence} analyzes the convergence of the obtained moment dynamics for certain systems and Section~\ref{sec:alg} discusses optimization problems, both for obtaining the reduced order models and for the reconstruction of a distribution from the moments. Section~\ref{sec:example} presents numerical examples and finally the conclusions are given in Section~\ref{sec:conclusion}.

\section{Background}\label{sec:background}
This section introduces some background material and also sets up notation used throughout the rest of the paper. 
To this end, by $\| \cdot \|$ we denote the vector $2$-norm and the corresponding induced matrix norm, and by $\| \cdot \|_{\Ltwo{X}}$ and $\| \cdot \|_{\Linf{X}}$ we denote the $L_2$ norm and the $L_\infty$ norm, respectively, for functions defined on $X$. Finally, $\mathbf C^1(X)$ denotes the set of continuously differentiable functions defined on $X$.

\subsection{From moments to distribution}\label{subsec:inverse}
Given a compact set $K\!\subset\!\mR^d$ and a family of kernel functions $\phi_k\!\in\! \mathbf C^1(K)$, $k\!\! =\!\! 1, \ldots, M$, the corresponding moments of a nonnegative measure $d\mu\!\in\! {\mathcal M}_+(K)$ are defined as
\begin{align}\label{eq:moments}
m_k := \int_K \phi_k(x)d\mu(x), \text{ for } k =1,\ldots M.
\end{align}
The problem of computing the moments given a measure is straight forward and requires little attention, whereas the inverse problem of recovering a nonnegative measure $d\mu$ from a sequence of numbers $m := (m_1, \dots, m_M)$ is a classical problem in mathematics \cite{
akhiezer1965classical, krein1977themarkov, lasserre2009moments}.
Although this inverse problem is in general ill-posed and there may be an infinite family of solutions, the set of moments still gives valuable macro-scale information about the distribution.
For example they can be used to give an estimate of the distribution with a resolution that depends on the kernel functions chosen and the accuracy of the moments, or to bound the mass of the measure in a given region \cite{karlsson2013uncertainty, marzetta1984power}.  
From the perspective of multi-agent systems this means that, e.g., in an evacuation scenario, we could answer questions regarding bounds or estimates on the number of individuals that are located in a certain area.


Moment problems also occur in many application areas, such as spectral estimation \cite{stoica2005spectral}, optimal control \cite{hernandez1996linear, gaitsgory2009linear}, and modeling the distribution of stochastic processes in, e.g., a chemical plant \cite{singh2011approximate} or an electrical or mechanical system \cite{ghusinga2017approximate}. Polynomial moments have also been used in the literature on collective leader-follower problems for crowd control \cite{yang2015shaping}, as they can be used to achieve polygon shapes \cite{milanfar1995reconstructing}.

\subsection{The logarithmic norm}
For a linear system $\dot x=Ax$, the spectral abscissa of the matrix $A$ can be used to determine stability and gives bounds on how sensitive the system is to perturbations. These concepts can be generalized to nonlinear systems by using the so-called logarithmic norm \cite{soderlind1984nonlinear, soderlind2006logarithmic}, \cite[Sec.~II.8]{desoer1975feedback}.
The logarithmic $2$-norm of a matrix $A \in \mathbb{R}^{d\times d}$ is defined as
\begin{equation}\label{eq:DefLN}
  \nu[A] := \lim _{h\to 0^+}  \frac{\|I+hA\|-1}{h}.
\end{equation}
It can be easily shown that $\nu[\cdot]$ is a convex function \cite[p.~31]{desoer1975feedback}.
Moreover, note that $\nu[A] = \lambda_{\rm max}( A+A^*)/2$, where $\lambda_{\rm max}(\cdot)$ is the maximal eigenvalue of a matrix \cite[p.~33]{desoer1975feedback}.

The standard matrix norm bound of the matrix exponential $\|\exp(At)\|\le \exp(\|At\|)$ is often
too conservative. However, the logarithmic norm allows a tighter bound by distinguishing between forward and reverse time.

\begin{lemma}[{\cite[Prop. 2.1]{soderlind2006logarithmic}}]\label{lem:expBound}
  Let $A \in \mathbb{R}^{d \times d}$, then
  \[
  \|e^{At}\| \leq e^{t \nu[A]}\quad  \mbox{ for all } t \geq 0.
  \]
\end{lemma}

In particular, when $\nu[A]$ is negative, the bound in Lemma~\ref{lem:expBound} is clearly less conservative than using any matrix norm.


The logarithmic norm can be extended to nonlinear dynamical systems $\dot{x} = f(x)$. To this end, consider the nonlinear mapping $f: K \subset \mR^d \to \mR^d$.
If $f$ is Lipschitz on $K$, then the least upper bound Lipschitz constant of $f$ is defined by
\begin{equation}\label{Def:L}
    L[f] := \sup_{u,v \in K, u \neq v} \frac{\|f(u)-f(v)\|}{\|u-v\|}.
  \end{equation}
Accordingly we define the least upper bound logarithmic Lipschitz constant as
\begin{equation}\label{eq:DefNonLN}
  M[f] := \lim _{h\to 0^+}  \frac{L[I+hf]-1}{h},
\end{equation}
where $I+hf$ denotes the mapping $x \mapsto x+hf(x)$. 
This logarithmic Lipschitz constant $M[f]$ is the nonlinear generalization of the logarithmic norm \eqref{eq:DefLN}, and if $f$ is a $\mathbf C^1$ mapping with a compact and convex domain, then $M[f]$ can be computed by the logarithmic norm of the Jacobian of $f$ \cite[p. 672]{soderlind1984nonlinear}.

\begin{lemma}\label{lem:LNormforC1}
If the function $f: K \subset \mR^d \to \mR^d$ is continuously differentiable and domain $K$ is convex and compact, then
\begin{equation*}
  M[f]=\sup_{x \in K} \nu[\nabla f(x)].
\end{equation*}
\end{lemma}
This lemma gives an alternative way to compute the least upper bound logarithmic Lipschitz constant $M[f]$ and will be used for deriving the error bounds in Section~\ref{subsec:interaction}.

\section{Representing multi-agent systems by moments}\label{sec:moment_rep}
Consider a multi-agent system consisting of $N$ identical agents, and let $x_i(t)\in \mR^d$ denote the state of agent $i$ at time $t$ for $i=1,\ldots, N$.
Throughout we will assume that every agent $x_i(t)$ belongs to the compact set $K\subset \mR^d$ for $t \in [0, T]$.
The distribution of the agents can be described in a concise way by
a nonnegative measure $d\mu_{t}\in {\mathcal M}_+(K)$ as
\begin{equation}\label{eq:moment_rep}
d\mu_{t}(x) = \frac{1}{N}\sum_{i=1}^N \delta(x - x_i{(t)}) dx.
\end{equation}
This occupation measure is a time-dependent distribution which conveys all information about the current states of the agents in the system.
We will use an approximation of this distribution in order to avoid having to compute the dynamics of each individual agent, which would be too expensive when the number of agents $N$ is large.


Let $\phi_k\in \mathbf C^1(K)$, for $k=1,\ldots, M$, be a set of kernel functions. The corresponding moments of distribution \eqref{eq:moment_rep} are then defined by
\begin{equation}\label{eq:def_moment}
m_k(t)=\int_K \phi_k(x)d\mu_{t}(x)=\frac{1}{N}\sum_{i=1}^N \phi_k(x_i(t))
\end{equation}
for $k=1,\ldots, M$.
In order to capture the collective behavior composed by all the individuals,
we investigate the dynamics of $d\mu_{t}(x)$ by considering approximate dynamics of the finite set of moments $\{m_k(t)\}_{k = 1}^M$. The approximated dynamics is then used to estimate the moments at a given time, and the occupation measure representing the particle distribution can be reconstructed accordingly by solving a moment matching problem.
%
In the following subsections we show how the dynamics of these moments can be approximated for different kinds of systems, and also derive bounds for the approximation errors.

\subsection{Modeling basic systems of agents}\label{sec:basicSystems}
We start with deriving the moment dynamics for systems where the dynamics of each individual is governed only by a spatial vector field.
The main purpose of this is to illustrate the theory, but it is also applicable to some applications such as crowd evacuation in a domain with obstacles \cite{ge2000new} and movement analysis for a particle accelerator \cite{wiedemann2015particle}.
In the next subsections we extend this framework to more general multi-agent systems.

Let the dynamics of each individual be governed by
\begin{equation}\label{Eq:BasicModels}
\dot x_i(t)=f(x_i(t)), \qquad i = 1, \ldots, N,
\end{equation}
where $f$ is Lipschitz continuous on $K$.
Correspondingly, the dynamics of the moments 
satisfies
\begin{align}
\dot m_k(t) & = \frac{1}{N} \! \sum_{i=1}^N \frac{d \phi_k(x_i(t))}{d t} = \frac{1}{N} \! \sum_{i=1}^N \frac{\partial \phi_k(x_i(t))}{\partial x_i(t)}f(x_i(t)) \nonumber \\
&=\int_{x\in K} \frac{\partial \phi_k(x)}{\partial x}f(x)d\mu_{t}(x). \label{eq:MomentSysBasic}
\end{align}
Similar to the Frobenius-Perron and Koopman operator frameworks, the nonlinear dynamics is lifted to an infinite dimensional linear dynamics in terms of measures \cite{budivsic2012applied, lasota2013chaos, mauroy2016global, rowley2017model}. However, these operators are typically used for analyzing the dynamics of one system, whereas we here utilize the lifting to express the moment dynamics as a linear function of the representing measure.

If the function $(\partial_x \phi_k) f(x)$ is well approximated by a linear combination $\sum_{\ell = 1}^{M}\! a^k_{\ell}\phi_\ell(x)$, where $a^k_{\ell}\!\!\!\in\! \mathbb{R}$ are some coefficients, then by the linearity of the integral and definition \eqref{eq:def_moment}, the dynamics of moment system \eqref{eq:MomentSysBasic} is approximated by
\[
\dot{m}_k(t)=\!\!\int_{x\in K} \frac{\partial \phi_k}{\partial x}fd\mu_{t}\approx\!\! \int_{x\in K} \sum_{\ell = 1}^{M} a^k_{\ell}\phi_\ell d\mu_{t}=\!\!\sum_{\ell = 1}^{M} a^k_{\ell}m_\ell(t).
\]
Thus the overall system can be approximated by the linear system
\begin{equation}\label{eq:ApproSysBasic}
\dot{\overline{m}}(t) =
\begin{bmatrix}
a^1_{1} & \cdots & a^1_{M} \\
\vdots  &        & \vdots  \\
a^M_{1} & \cdots & a^M_{M}
\end{bmatrix}
\! \!
\begin{bmatrix}
\overline m_1(t) \\ \vdots \\ \overline m_M(t)
\end{bmatrix}
\! =: A\overline m(t),
\end{equation}
where $\overline m(t)=(\overline m_1(t), \dots, \overline m_M(t))\in  \mR^{M}$ is a vector of the approximate moments.
The accuracy of the model \eqref{eq:ApproSysBasic} and the amount of information it carries about the multi-agent nonlinear systems \eqref{Eq:BasicModels} depend on the number of moments and on the selected kernel functions $\phi_k$ (see Section~ \ref{subsec:inverse}).
Denote the approximation error of function $(\partial_x \phi_k)f(x)$ by
\begin{equation}\label{eq:errorBasic}
\varepsilon_k(x) := \frac{\partial \phi_k(x)}{\partial x} f(x) - \sum_{\ell=1}^M  a^k_{\ell}\phi_\ell(x).
\end{equation}
Before stating the first main result, we recall the basic assumption that allows for this derivation.
\begin{assumption}\label{ass:inK}
Assume that every agent $x_i(t)$ belongs to the compact set $K\subset \mR^d$ for $t \in [0, T]$.
\end{assumption}
This assumption holds for many cases of interest.  
For example for a given initial particle distribution and time $T$, such a set $K$ can always be found provided that the system $f$ is globally Lipschitz \cite[Thm.~2.3]{khalil1992nonlinear}.
Another example is when each subsystem is equibounded (see, e.g., \cite[Def.~1.5.1]{kato1996stability}), where the assumption holds for $T\! =\! \infty$. Further comments on this can be found in Remark~\ref{rem:whenKbounded}, in the end of this section.
The following result gives the error bound for moment system \eqref{eq:ApproSysBasic}.
\begin{thm}\label{thm:ErrorBoundBasic}
Assume that Assumption~\ref{ass:inK} holds. Let $\overline{m}(t)$ and $m(t)$ be the solutions of the approximate moment dynamics \eqref{eq:ApproSysBasic} and the true moment dynamics \eqref{eq:MomentSysBasic} respectively.
Then for $t\in [0,T]$, the difference of the two solutions $\Delta m(t)=m(t)-\overline{m}(t)$ is bounded by
\[
  \|\Delta m(t)\| \leq \|\Delta m(0)\| e^{t\nu [A]} + \frac{e^{t \nu [A]}-1}{\nu [A]} \sqrt{\sum_{k=1}^M \max_{x\in K} \varepsilon_k(x)^2},
\]
if $\nu [A] \neq 0$, and by
\[
  \|\Delta m(t)\| \leq \|\Delta m(0)\| + t\sqrt{\sum_{k=1}^M \max_{x\in K} \varepsilon_k (x)^2}
\]
if $\nu [A] = 0$, where  $\varepsilon_k(x)$ is defined in \eqref{eq:errorBasic}.
\end{thm}

\begin{proof}
See appendix~\ref{sec:app1}.
\end{proof}

%

For $\nu[A] < 0$, we get the following time-independent bound.
\begin{cor}\label{cor:ErrorBoundBasic}
Under the conditions in Theorem~\ref{thm:ErrorBoundBasic}, if $\nu[A]<0$ and $\Delta m(0)=0$, then for $t\in [0,T]$ the error is bounded by
  \[
        \|\Delta m(t)\| \leq -\frac{1}{\nu [A]} \sqrt{\sum_{k=1}^M \max_{x\in K} \varepsilon_k (x)^2}.
  \]
\end{cor}

\begin{remark}
Note that Theorem~\ref{thm:ErrorBoundBasic} can be generalized, using any vector norm $\|  \cdot \|_* $ and the corresponding logarithmic norm $\nu_{*}[A]$, see, e.g. \cite{soderlind2006logarithmic}. In fact, Theorem~\ref{thm:ErrorBoundBasic} is still valid if one simply changes the vector norms and the logarithmic norms accordingly, and also changes $\sqrt{\sum_{k} \max_{x\in K} \varepsilon_k (x)^2}$ to %
$\Big\| \, \big[    \| \varepsilon_k (x)\|_{\Linf{K}}  \big]_{k=1}^M  \, \Big\|_* $.
\end{remark}

Theorem~\ref{thm:ErrorBoundBasic}
indicates that the accuracy of the moment-based model depends not only on the instantaneous precision in the approximation of the dynamics, as given by $\varepsilon_k$, but also on the propagation of the approximation error in time.  The logarithmic norm of the resulting system matrix $A$ gives a bound on this propagation.
This implies that a trade-off between accuracy and stability of the approximate moment dynamics needs to be taken into account (see also Section~\ref{subsec:1D_example}).

\subsection{Modeling multi-agent systems with interactions}\label{subsec:interaction}
In multi-agent systems, besides a spatial vector field, the interactions between each pair of individuals often play an essential role in its collective behavior \cite{altafini2015predictable, yang2015shaping}.
To account for this in the model, consider agents governed by the dynamics
\begin{equation}\label{eq:nonInteraction}
\dot x_i(t)=\frac{1}{N}\sum_{j=1}^N g(x_i(t),x_{j}(t)),
\end{equation}
where $g (\cdot, \cdot)$ is a Lipschitz continuous function in both arguments.
Then the exact moment dynamics 
is given by
\begin{align}
\dot m_k(t) & = \frac{1}{N}\sum_{i=1}^N \frac{d \phi_k(x_i(t))}{d t} \nonumber \\
& = \frac{1}{N}\sum_{i=1}^N \frac{\partial \phi_k(x_i(t))}{\partial x_i(t)}\frac{1}{N}\sum_{j=1}^N g(x_i(t),x_{j}(t)) \nonumber \\
& = \int_{x\in K} \int_{y\in K}\frac{\partial \phi_k(x)}{\partial x}g(x,y)d\mu_{t}(x)d\mu_{t}(y) \label{eq:true_dyn}.
\end{align}
Similarly to the previous case, provided that we can approximate the function $(\partial_x \phi_k(x)) g(x,y)$ in terms of the kernel functions $\{\phi_j(x)\phi_\ell(y)\}_{j,\ell = 1}^M$ on $(x, y) \in K^2$, i.e., 
\begin{equation}\label{eq:errorInter0}
\frac{\partial \phi_k(x)}{\partial x}g(x,y) \!\approx\!\! \sum_{\ell, j=1}^M \!\Belement^k_{\ell, j} \phi_\ell(x)\phi_j(y), \mbox{ for } (x,y)\in K^2,
\end{equation}
by the linearity of the integral and \eqref{eq:def_moment}, the moment dynamics $\dot{m}_k(t)$ can be approximated as
\begin{align}\label{eq:approx_dyn}
\dot{\overline{m}}_k(t) = \sum_{\ell, j=1}^M \Belement^k_{\ell, j} m_\ell(t) m_j(t) = m(t)^T B_k{m}(t),
\end{align}
where $B_k = [b_{\ell,j}^k]_{\ell, j = 1}^M$.
The general nonlinear interaction \eqref{eq:nonInteraction} is thus approximated by a simple quadratic system, and the approximation error in \eqref{eq:errorInter0} is denoted by
\begin{equation}\label{eq:errorInter}
\varepsilon_k (x,y) :=\frac{\partial \phi_k(x)}{\partial x}g(x,y) - \! \sum_{\ell, j=1}^M \Belement^k_{\ell, j} \phi_\ell(x)\phi_j(y). 
\end{equation}
To bound the approximation error of the moment dynamics, it can be shown that if the approximate moment trajectory is contained in the compact and convex set $D$, then the logarithmic norm of the system is bounded by
\begin{equation}\label{eq:BetaInter}
  \Bbound=\max_{m \in D} \sum_{\ell=1}^M  \, \nu[ m_{\ell}\widetilde{B}_{\ell} ],
\end{equation}
where $\widetilde{B_\ell}:=[\Belement^i_{\ell,j}+\Belement^i_{j,\ell}]_{i,j=1}^M$.
A natural choice is to let $D$ be the set of all possible moment sequences for agents in $K$, i.e.,  $D=\Ccal := \left\{ m \in \mR^{M} \mid \eqref{eq:moments}, \;  d\mu\in {\mathcal M}_+(K), \; \int_K d\mu = 1 \right\}$. 
\begin{assumption}\label{ass:inD}
Assume that the approximate moment trajectory $\overline m(t)$, i.e., the solution to \eqref{eq:approx_dyn}, belongs to $D$ for $t\in [0, \overline{T}]$.
\end{assumption}

This gives the following theorem. 

\begin{thm}\label{thm:ErrorBoundInter}
Assume that Assumption~\ref{ass:inK} and Assumption~\ref{ass:inD} hold, and that the support set $K$ is convex.
Let $m(t)$ be the solution of the true moment dynamics \eqref{eq:true_dyn}
and let $\overline{m}(t)$ be the solution of the approximate dynamics \eqref{eq:approx_dyn}.
Then for $t\in [0, \min(T, \overline{T})]$, the norm of the error trajectory $\Delta m(t)=m(t)-\overline{m}(t)$ is bounded by
\[
  \|\Delta m(t)\| \leq \|\Delta m(0)\| e^{\Bbound t} + \frac{e^{\Bbound  t}-1}{\Bbound} \sqrt{\sum_{k=1}^M \max_{x,y\in K} \varepsilon_k (x,y)^2},
\]
if $\Bbound\neq 0$, and by
\[
  \|\Delta m(t)\| \leq \|\Delta m(0)\| + t\sqrt{\sum_{k=1}^M \max_{x,y\in K} \varepsilon_k (x,y)^2},
\]
if $\Bbound= 0$, where  $\varepsilon_k (x,y)$ is defined in \eqref{eq:errorInter} and $\beta$ in \eqref{eq:BetaInter}.
\end{thm}

\begin{proof}
See appendix~\ref{sec:app2}.
\end{proof}

Moreover, similar to Corollary~\ref{cor:ErrorBoundBasic}, for a system with negative factor $\beta$
a time-independent bound can easily be derived.

\subsection{Modeling multi-agent systems with leaders}
In some applications the agents in the multi-agent system may not be identical but instead be 
heterogeneous.
For example,
 some agents may be equipped with different on-board sensors, have different movement capability, or have access to global information.
The most commonly used setup to specify such architectures is the so-called leader-follower setting \cite{gustavi2010sufficient, yang2015shaping},
to which we devote the following subsection.


In the leader-follower setting a few agents, called leaders, are distinguished from the remaining  agents, called followers. Let the dynamics of the leaders be governed by
\[
   \dot{y}_j(t) = f_L(y_j(t), u_j(t)),
\]
for $j= 1, \dots, N_L$. Here  $y_j \in K$ is the state of leader $j$, and $u_j$ is a control signal which can involve global information of the system, e.g., states of all agents or control goal of the overall system. Furthermore, the dynamics of the followers are given by
\begin{equation}\label{eq:ModelsWithLeaders}
    \dot{x}_i (t) = h(x_i(t), Y(t)),
\end{equation}
where $h: \mR^d \times \mR^{N_L} \to \mR^d$ is continuously differentiable, and $Y=\begin{pmatrix}
y_1, \dots, y_{N_L}
\end{pmatrix}.$
If
it is possible to approximate the functions
$(\partial_x \phi_k(x))h(x, Y)$ in terms of the separable sum
$\sum_{\ell=1}^M {c}^k_{\ell}(Y) \phi_\ell(x)$, where ${c}^k_{\ell}\in \mathbf{C}^1(K^{N_L})$, then the dynamics of moment $k$ can be approximated as
\begin{align}\label{eq:MomentSysLeader}
  \dot{m}_k(t)
  =&\int_{x\in K} \frac{\partial \phi_k(x)}{\partial x}  h(x,Y) d\mu_{ t}(x)\\
  \approx& \sum_{\ell=1}^M {c}^k_{\ell}(Y)  m_\ell(t). \nonumber
\end{align}
The functions ${c}^k_{\ell}\in \mathbf{C}^1(K^{N_L})$ represent the dependence on the states of all leaders, and the approximation results in the error
\begin{equation}\label{eq:errorLeader0}
\varepsilon_k (x,Y)=\frac{\partial \phi_k(x)}{\partial x} h(x, Y) - \sum_{\ell=1}^M {c}^k_{\ell}(Y) \phi_\ell(x).
\end{equation}
Therefore the dynamics of the estimated moments is written accordingly as
\begin{equation}\label{eq:ApproSysLeader0}
\dot{\overline{m}}(t)
=\!
\begin{bmatrix}c^1_{1}(Y)&\!\!\cdots\!\!&c^M_{1}(Y)\\\vdots&&\vdots\\c^M_{1}(Y)&\!\!\cdots\!\!& c^M_{M}(Y)\end{bmatrix}
\!\!\!\begin{bmatrix}\overline{m}_1(t)\\\vdots\\ \overline{m}_M(t)\end{bmatrix}
\!\!=: \! C(Y)\overline{m}(t),
\end{equation}
where $\overline m(t)=(\overline m_1(t), \dots, \overline m_M(t))\in \mR^{M}$.
Note that since ${c}^k_{\ell}(Y)\in \mathbf C^1$,
the solution $\overline{m}(t)$ to system  \eqref{eq:ApproSysLeader0} exists and is unique in any time interval once the moment trajectory evolves in the compact set $\Ccal$.

\begin{remark}\label{rmk:L2giveExpression}
The existence of $\mathbf C^1$-functions ${c}^k_{\ell}(Y)$ in approximation \eqref{eq:errorLeader0} strongly relies on the particular approximation used. For instance, $L_2$ approximation (see \ref{sec:alg}-A) can provide an analytical expression of ${c}^k_{\ell}(Y)$, where ${c}^k_{\ell} \in \mathbf C^1$ if $h\in \mathbf C^1$.
\end{remark}

Let $\tau := \max_{Y \in K^{N_L}} \{\nu[C(Y)]\}$. For any $Y(t) \in K^{N_L}$, by Theorem~\ref{thm:ErrorBoundBasic} the error of the approximate moment $\Delta m(t)=m(t)-\overline{m}(t)$ is bounded by
\begin{equation}\label{eq:leader_bounds}
  \|\Delta m(t)\| \!\leq \!\|\Delta m(0)\| e^{\tau t}\!+\! \frac{e^{ \tau t}\!-\!1}{\tau}\! \sqrt{\sum_{k=1}^M
  \max_{\substack{x\in K \\  Y \in K^{N_L}}} \!\!\!\varepsilon_k (x,Y)^2},
\end{equation}
if $\tau \neq 0$, where
$m(t)$ and $\overline{m}(t)$  are the solutions of the systems \eqref{eq:MomentSysLeader} and  \eqref{eq:ApproSysLeader0}, respectively.

\begin{remark}\label{rkm:controlProblem}
By means of model \eqref{eq:ApproSysLeader0}, the \emph{tracking control problem} for the approximate moments can be formulated as follows: find a control $u_j$, for $j=1, \dots, N_L$, such that for the closed-loop system
\begin{align*}
  \dot{y}_j&=f_L(y_j, u_j), \\
  \dot{\overline{m}} &= C(Y) \overline{m},
\end{align*}
the moments $\overline{m}(t)$ track a reference signal $m_r(t)$.
\end{remark}


In the case with multiple leaders, the domain of the function $C(Y)$ is high-dimensional and numerical computations are intractable. In order to handle this case we introduce additional assumptions on the follower dynamics, i.e., we assume that all the leaders have an identical effect on the followers.
\begin{assumption}\label{Ass1}
In the followers' dynamics \eqref{eq:ModelsWithLeaders}, the impact of each leader is additive and governed by a same law, i.e.,  the function $h$ admits the form $h(x, Y)=\sum_{j=1}^{N_L} \eta(x, y_j)$.
\end{assumption}

Under Assumption~\ref{Ass1}, the matrix-valued function $C(Y)$ in dynamics \eqref{eq:ApproSysLeader0} can be rewritten into a separable sum as
\begin{equation*}
\dot{\overline{m}} (t)
=\sum_{j=1}^{N_L} \Gamma(y_j) \overline m(t),
\end{equation*}
where matrix $\Gamma(y)\!=\![\gamma_\ell^k(y)]_{k,\ell=1}^M$ with entries obtained from the approximation $(\partial_x \phi_k(x)) \eta(x, y) \!\approx\! \sum_{\ell=1}^M \gamma^k_{\ell}(y) \phi_\ell(x)$, for any $x, y \in K$.

Next, in order to avoid the variable dependence of $\gamma^k_{\ell}(y)$ on $y$, we introduce a basis $\{\psi_r(y)\}_{r=1}^{M_L} \subset \mathbf C(K)$. Suppose that we have a good approximation for the two-variable function $(\partial_x \phi_k(x)) \eta(x, y)$ by the functions $\{ \psi_r(y)\phi_{\ell}(x) \}$, i.e., that we have a small approximation error
\begin{equation}\label{eq:ApproSysLeader1}
\varepsilon_k(x,y) := \frac{\partial \phi_k(x)}{\partial x} \eta(x, y) - \sum_{\ell=1}^M \sum_{r=1}^{M_L} \gamma^k_{\ell,r}\psi_r(y) \phi_\ell(x).
\end{equation}
Then the true moment dynamics  \eqref{eq:MomentSysLeader} can be approximated as
\begin{equation}\label{eq:ApproSysLeader2}
\dot{\overline{m}} (t)
=\sum_{j=1}^{N_L}  \sum_{r=1}^{M_L}  \psi_r(y_j)\Gamma_r \overline m(t),
\end{equation}
where $\Gamma_r:=[\gamma_{\ell,r}^k]_{k,\ell=1}^M$.
In the estimate dynamics \eqref{eq:ApproSysLeader2}, the approximation coefficients $\gamma_{\ell,r}^k$ are independent of  the leaders' positions and only depend on the functions $\eta(x,y)$,
and the kernels selected, which means that they can be computed off-line.
 This gives a similar error bound as that in \eqref{eq:leader_bounds}, namely
\begin{equation}\label{eq:leader_bounds_simplified}
  \|\Delta m(t)\| \!\leq \! \|\Delta m(0)\| e^{\tau t} \!+ \!N_L \frac{e^{ \tau t} \!- \!1}{\tau} \!\sqrt{ \sum_{k=1}^M
  \max_{x,y \in K} \varepsilon_k(x,y)^2},
\end{equation}
if $\tau \!\neq\! 0$, where $\varepsilon_k(x,y)$ is the approximation error in \eqref{eq:ApproSysLeader1} and
\begin{equation}\label{eq:tau}
\tau=\max_{Y \in K^{N_L}} \left\{ \nu\left[\sum_{j=1}^{N_L}  \sum_{r=1}^{M_L}  \psi_r(y_j)\Gamma_r\right] \right\}.
\end{equation}


Note that one is free to choose the basis $\{\psi_r(y)\}_{r=1}^{M_L}$  from any family of continuous functions in $\mathbf C(K)$, and thus we can select them differently from $\{\phi_k\}_k$.



%

\subsection{Combining basic dynamics, interaction, and leader terms}
Now, consider the multi-agent system where each agent is governed by
\begin{equation}\label{eq:dynamics_combined}
\dot x_i(t)= f(x_i) + \frac{1}{N}\sum_{j=1}^N g(x_i,x_{j}) + \sum_{j=1}^{N_L} \eta(x_i, y_j),
\end{equation}
where $f$ is a spatial vector field as in \eqref{Eq:BasicModels}, $g$ is the interaction term as in \eqref{eq:nonInteraction}, and $h(x, Y) = \sum_{j=1}^{N_L} \eta(x, y_j)$ is the leaders influence as in
Assumption~\ref{Ass1}.
To derive a moment-based model for this system one can perform the function approximations to minimize the errors \eqref{eq:errorBasic}, \eqref{eq:errorInter}, and
\eqref{eq:ApproSysLeader1}
separately.
Introducing the notation $\mB(m):=(m^T B_1^Tm, \dots, m^T B_M^Tm)$, the resulting moment dynamics would be
\begin{equation}\label{eq:approx_dyn_all}
\dot{\overline{m}}(t)
= A\overline{m}(t) + \mB(\overline{m}) + \sum_{j=1}^{N_L}  \sum_{r=1}^{M_L}  \psi_r(y_j)\Gamma_r\overline{m}(t).
\end{equation}

Similar to Theorems~\ref{thm:ErrorBoundBasic} and \ref{thm:ErrorBoundInter}, the error of the moments can be bounded as in the following corollary. Denote  $\varepsilon_k(x,y,Y)\! := \!\varepsilon_k^{f}(x)\! +\!\varepsilon_k^{g}(x,y) \!+ \!\sum_{j=1}^{N_L} \varepsilon_k^{\eta}(x,y_j)$, where the individual terms are defined by \eqref{eq:errorBasic}, \eqref{eq:errorInter} and \eqref{eq:ApproSysLeader1}, respectively.

\begin{cor}
Under the conditions in Theorem~\ref{thm:ErrorBoundInter},
for $t\in [0, \min(T, \overline{T})]$, the norm of the error trajectory $\Delta m(t)=m(t)-\overline{m}(t)$ is bounded by
\[
  \|\Delta m(t)\| \leq \|\Delta m(0)\| e^{\zeta t} + \frac{e^{\zeta  t}-1}{\zeta} \! \sqrt{\sum_{k} \!
  \max_{\substack{x,y\in K, \\ Y \in K^{N_L}} }
  \varepsilon_k(x,y,Y)^2},
\]
 if $\zeta \!\!:=\!\!\nu[A]\!+\! \beta\! +\! \tau \!\!\neq\! 0$, where $\beta$ is defined in  \eqref{eq:BetaInter} and $\tau$ in \eqref{eq:tau}.
\end{cor}

The following remark gives a sufficient condition for when Assumption~\ref{ass:inK} holds for systems with dynamics given by \eqref{eq:dynamics_combined}.

\begin{remark}\label{rem:whenKbounded}
Assume that $f \in \mathbf C^1(\mR^d)$ with bounded Jacobian $[\partial f /\partial x ]$, and that $x = 0$ is a globally exponentially stable equilibrium point of the system $\dot{x} = f(x)$. If the terms $g(\cdot, \cdot)$ and $\eta(\cdot, \cdot)$ are globally bounded, then there is a compact region $K$ such that Assumption~\ref{ass:inK} holds for $T = \infty$.
\end{remark}

The remark can be shown by, e.g., using the input-output stability result \cite[Thm.~4.13]{khalil1992nonlinear}, where we identify the interaction and leader terms as the input signal, and the full state as the output signal.



\section{Convergence for monomial kernels with $d=1$}\label{sec:Convergence}

In this section we will consider the
setup where the agents are governed by the basic nonlinear model (\ref{Eq:BasicModels})
 and give sufficient conditions for when the approximate moment dynamics converge to the dynamics of the true moments as the number of kernel functions goes to infinity, i.e., when $M \to \infty$.
In particular, we consider systems defined on the interval $[-1,1]$, and we select the kernel functions to be the monomials,%
\footnote{For convenience of notation, we will in this section index the kernels from zero, i.e., $k=0, 1, \ldots, M$.}
$\phi_k(x)=x^{k}$ where $k=0,1,\ldots, M$. For this case, we will give conditions on the system dynamics which guarantees the convergence of the moment dynamics.
Note that the number of functions
$\big\{f(x)\partial_x \phi_k(x)\big\}_{k=1}^M$ to be approximated increases with $M$ and hence the convergence of uniform approximation bounds do not necessarily follow even if the span of the family of kernel functions is dense in the space $\mathbf C(K)$.

First, let $f\in \mathbf C[-1,1]$ be a continuous function on the interval $ [-1,1]$, and let
$\mathbf{E}_n(f)$ be the error of the best $L_\infty$ approximation of $f$ by polynomials up to degree $n$, i.e.,
\begin{equation*}
  \mathbf{E}_n(f)= \min_{p \in \mathcal{P}_n}\|f(x)-p(x)\|_{\Linf{[-1,1]}},
\end{equation*}
where $\mathcal{P}_n$ is the set of all polynomials with degree at most $n$.
When $f$ is a monomial, the approximation errors can be bounded in terms of the parameters
  \begin{equation}\label{eq:p_k,n}
  P_{k,n}:= \frac{1}{2^{k-1}} \sum_{j > \frac{n+k}{2}}^{k} \binom{k}{j}.
  \end{equation}
In particular, the following lemma states that for odd $M$, even $n$, and all $0 \leq k \leq M$, the error $\mathbf{E}_n(x^k)$  is bounded by $P_{M,n}$.

\begin{lemma}\label{thm:P_k,nMontonicity}
  Let $n$ and $M$ be positive integers. For any integer $k$ such that $0\leq k \leq 2M$,
  \[
    \mathbf{E}_{2n}(x^k)\leq P_{2M-1,2n}.
  \]
\end{lemma}

\begin{proof}
See Appendix~\ref{sec:Apx_Pf_montonicity}.
\end{proof}

By using this lemma
one can obtain uniform bounds on the approximation errors of the moment dynamics. The following lemma shows that these errors
uniformly converge to zero as the number of moments goes to infinity.
\begin{lemma}\label{thm:Convergence_MomentSys_m}
  Let $f\in \mathbf C^3[-1,1]$ and let $\phi_k(x) = x^k$, then
  \begin{equation}\label{eq:thm:Convergence_MomentSys_m1}
   \max_{k=1, \dots, 4M} { M \mathbf{E}_{4M}\Big(f(x)\partial_x \phi_k(x)  \Big)} \to 0,
  \end{equation}
  as $M \to \infty$.
\end{lemma}
\begin{proof}
  See Appendix~\ref{sec:Apx_Pf_Convergence_dynamics}.
\end{proof}

The main result of this section can now be stated as follows.
\begin{thm}\label{thm:Convergence_MomentSys_2Norm}
Given $d\mu_t \in \mathcal{M}_+([-1,1])$ 
and $f\in \mathbf C^3[-1,1]$,
let $m^M\!\!\in\!\! \mathbb{R}^{4M\!+\!1}$ be the moments corresponding to the monomial kernels, i.e., $\phi_k(x)=x^{k}$ where $k=0,1,\ldots, 4M$, and let $\dot m^M$ be the dynamics of the true moment system given by \eqref{eq:MomentSysBasic}. Furthermore, let $\dot{\overline{m}}^M \!\in\! \mathbb{R}^{4M\!+\!1}$ be the estimated dynamics in \eqref{eq:ApproSysBasic} under the $L_\infty$ approximation. Then
  \begin{equation}\label{eq:thm_Convergence_MomentSys_2Norm}
    \lim_{M\to \infty} \|\dot{\overline{m}}^M -\dot m^M \| =0,
   \end{equation}
i.e., the norm of the the instantaneous error dynamics goes to zero as the number of moments goes to infinity.
\end{thm}


\begin{proof}
 Note that
 \begingroup
 \allowdisplaybreaks
 \begin{align*}
 \|\dot{\overline{m}}^M -\dot m^M \|& = \left(\sum_{k=1}^{4M} \left( \int_K \varepsilon_k (x) d\mu_t(x) \right)^2\right)^{1/2} \\
 &\leq \left(\sum_{k=1}^{4M} \Big(  \mathbf{E}_{4M}(kx^{k-1} f(x))  \Big)^2\right)^{1/2} \\
 &\leq  2 M^{1/2} \max_{k=1, \dots, 4M} {\mathbf{E}_{4M}(kx^{k-1} f(x) )},
 \end{align*}
 \endgroup
 which by Lemma~\ref{thm:Convergence_MomentSys_m} goes to $0$ as $M \to \infty$.
\end{proof}

When the smoothness of $f$ increases, the
 convergence rate of the errors 
 is improved 
  as stated in the next corollary. Its proof follows from Corollary~\ref{thm:Convergence_MomentSys_m^l} in Appendix~\ref{sec:Apx_Pf_Convergence_dynamics}.

\begin{cor}\label{thm:Convergence_MomentSys_2Norm_m^l}
  If the function $f$ in Theorem~\ref{thm:Convergence_MomentSys_2Norm} is $(\ell+3)$-continuously differentiable, i.e., $f\in \mathbf C^{\ell+3}[-1,1]$, then Theorem~\ref{thm:Convergence_MomentSys_2Norm} holds with the equation (\ref{eq:thm_Convergence_MomentSys_2Norm}) replaced by
    \begin{equation*}
    \lim_{M\to \infty} M^{\ell} \, \|\dot{\overline{m}}^M -\dot m^M \| =0.
  \end{equation*}
\end{cor}

The convergence result in Theorem~\ref{thm:Convergence_MomentSys_2Norm} and Corollary~\ref{thm:Convergence_MomentSys_2Norm_m^l} can be applied for general intervals $[a,b]$,
provided that the kernel functions are selected accordingly, i.e.,  
 as the normalized and translated monomials
\[
\bar \phi_k(x) = \left( \frac{2x - (b + a)}{b - a} \right)^k,\quad  \mbox{ for }k=0,\ldots, M.
\]
This corresponds to the normalization $\|\bar \phi_k(x)\|_{\Linf{[a,b]}}=1$. It should be noted that the corresponding convergence result in Theorem~\ref{thm:Convergence_MomentSys_2Norm} does not hold for arbitrary intervals without this normalization.

For example, even for the $\mathbf C^\infty$ function $f(x)=x^2$ the corresponding moment dynamics can not be approximated
if the length of the interval is larger than $4$. In fact, noting that $f \partial_x \phi_k=k x^{k+1}$,
the following lemma shows that the bound of the dynamics error in  (\ref{eq:thm_Convergence_MomentSys_2Norm}) diverges as $M\to \infty$.
\begin{lemma}\label{lem:limit_E2m+1,2m}
   For an interval $[a,b]$, the approximation error
   $\mathbf{E}^{[a,b]}_n(x^{n+1}):= \min_{p \in \mathcal{P}_n}\|x^{n+1}-p(x)\|_{\Linf{[a,b]}}$ satisfies
  \[\lim_{n \to \infty} \mathbf{E}_n^{[a,b]}(x^{n+1}) \to \infty,
   \]
   if the length of the interval $(b-a)>4$.
\end{lemma}
\begin{proof}
See Appendix~\ref{sec:Apx_Pf_limit_E2m+1,2m}.
\end{proof}
Loosely speaking, this restriction on the interval length can be viewed as an improper normalization of the kernel functions and gives rise to limitations of the approximation capacity of the monomials.
 This also provides an impelling reason for carefully considering the choice of kernel functions. For example considering Gaussian kernel functions or Chebyshev kernels, especially when a large region is considered.
It also highlights the need to characterize
what information that the moments carry about the measure, and thus how errors in the moments propagate to errors in the measure.

Finally, note that good approximation of the moment dynamics does not necessarily give good approximation of the moment trajectories. The accuracy of the moment trajectories also depends on the error propagation, as seen in Theorems \ref{thm:ErrorBoundBasic} and \ref{thm:ErrorBoundInter}. The topic on convergence of the moment trajectories will be left for a future study.


\section{Algorithms for moment-based modeling}\label{sec:alg}
In this section we discuss how the approximation of the dynamics in \eqref{eq:errorBasic}, \eqref{eq:errorInter} and
\eqref{eq:ApproSysLeader1} can be implemented. We will also consider how to reconstruct a nominal distribution from the estimate moments.
A high-level algorithm is given in Algorithm~\ref{alg:moment_method}. In particular, step 2 and 4 in this algorithm are outlined in more detail in the coming two subsections.


\renewcommand{\algorithmicrequire}{\textbf{Input:}}
\renewcommand{\algorithmicensure}{\textbf{Output:}}
\algsetup{indent=12pt}

\begin{algorithm}[b]
\caption{High-level algorithm for moment-based modeling of collective behavior.}
\label{alg:moment_method}
\begin{algorithmic}[1]
\REQUIRE Agent dynamics \eqref{eq:dynamics_combined}, i.e., $N$, $N_L$, $f$, $g$, and $\eta$.
\STATE Choose kernel functions $\{ \phi_k(x) \}_{k=1}^M$ and $\{ \psi_k(y) \}_{k=1}^{M_L}$.
\STATE \emph{Moment dynamics approximation}: compute the approximate moment dynamics \eqref{eq:approx_dyn_all}, by an appropriate minimization of the approximation errors $\varepsilon_k$ in \eqref{eq:errorBasic}, \eqref{eq:errorInter} and \eqref{eq:ApproSysLeader1}.
\STATE Solve the obtained ODE system \eqref{eq:approx_dyn_all} starting form the initial condition $m(0)$, or from an approximation thereof.
\STATE \emph{Reconstruction of the distribution}: reconstruct macroscopic properties of interest by solving a convex optimization problem \eqref{eq:reco_opt_prob_general}.
\end{algorithmic}
\end{algorithm}

\subsection{Moment dynamics approximation}\label{subsec:moment_dyn_approx}
First consider the problem of approximating the dynamics for basic systems of agents \eqref{Eq:BasicModels}, that is, systems without interactions. Note that the corresponding moment dynamics $[\dot m_k(t)]_{k=1}^M$ from \eqref{eq:MomentSysBasic}
is close to the dynamics of the linear system $\dot{ \overline m}(t)$ in  \eqref{eq:ApproSysBasic} if the error terms $\varepsilon_k (x)$ in \eqref{eq:errorBasic} are small on the support of the agents. An explicit bound on the difference $\Delta m(t)=m(t)-\overline m(t)$ is given in terms of the $L_\infty(K)$-norm of the approximation error \eqref{eq:errorBasic}, according to Theorem~\ref{thm:ErrorBoundBasic}.
Also note that the moment error $\Delta m$ not only depends on the instantaneous error in the dynamics, but also on the propagation of the error in time, which is governed by $\nu[ A ]$. Thus in the approximation procedure one needs to balance the trade-off between both factors. To this end, we introduce constraints on the logarithmic norm $\nu[ A ]$, resulting in the optimization problem
\begin{align}
   \min_{a^k_{\ell}} \quad  \quad& \sum_{k=1}^{M} \left\| \frac{\partial \phi_k(x)}{\partial x} f(x) - \sum_\ell a^k_{\ell}\phi_\ell(x) \right\|
   _{L_\infty(K)} \nonumber\\
  \text{subject to}  \quad &  \nu[A] \le \kappa, \label{eq:optimizationBasic}
\end{align}
where $\kappa\in \mathbb{R}$ is a tunable constant. This is a convex problem and there is a large number of numerical algorithms that can be applied. One way to solve it is by first discretizing the domain $K$ into a finite number of grid points, and then solve the resulting approximate problem with, e.g., CVX \cite{grant2008graph, cvx}. This is the method used to solve these problems in Section~\ref{sec:example}.

In the corresponding problems with interactions
\eqref{eq:nonInteraction} and leaders \eqref{eq:ModelsWithLeaders} we can also
approximate the moment dynamics by ODEs, and the errors in the approximate moments
can be bounded
according to  Theorem~\ref{thm:ErrorBoundInter} and  \eqref{eq:leader_bounds_simplified}, respectively.
Thus we can formulate corresponding optimization problems for the approximate moment dynamics \eqref{eq:approx_dyn} with interactions as
\begin{align}
   \min_{\Belement^k_{\ell, j}} \quad  & \sum_{k=1}^{M} \left\| \frac{\partial \phi_k(x)}{\partial x}g(x,y) - \! \sum_{\ell, j=1}^M \Belement^k_{\ell, j} \phi_\ell(x)\phi_j(y) \right\|
   _{L_\infty(K^2)} \nonumber \\
  \text{s.t.}  \quad&
  \max\big\{\nu[\widetilde{B}_{\ell}], \nu[-\widetilde{B}_{\ell}]\big\} \le \kappa_\ell,  \quad \ell=1, \dots, M, \label{eq:optimizationInter}
\end{align}
and for the approximate moment dynamics \eqref{eq:ApproSysLeader2} with leaders as
\begin{align}
   \min_{\gamma^k_{\ell, r}}  \;\; & \sum_{k=1}^{M} \left\| \frac{\partial \phi_k(x)}{\partial x} \eta(x, y) - \sum_{\ell=1}^M \sum_{r=1}^{M_L} \gamma^k_{\ell,r}\psi_r(y) \phi_\ell(x) \right\|
   _{L_\infty(K^2)} \nonumber \\
\text{s.t.}  \;\; &
\max\big\{\nu[\Gamma_\ell], \nu[-\Gamma_\ell]\big\}
\!\le\! \kappa_{\ell}, \, \ell\!=\!1, \dots, M_L. \label{eq:optimizationLeader}
\end{align}
Here the constants $\{\kappa_{\ell}\}_\ell$ give the bounds on the logarithmic norms
$\nu[ \widetilde{B}_{\ell} ]$ and $\nu[\Gamma_\ell]$, and can be selected
to bound the propagation of the error in time but at the expense of a possibly worse instantaneous error.

Another approximation method that could be used is to use the $L_2(K)$ (or $L_2(K^2)$) in the objective function.
This objective function cannot be directly interpreted and justified in terms of Theorem~\ref{thm:ErrorBoundBasic}, however, minimization of the $L_2$ norm also tends to make the largest deviations small. Further the objective functions in \eqref{eq:optimizationBasic}, \eqref{eq:optimizationInter}, \eqref{eq:optimizationLeader} (with $L_\infty$ norm replaced by squared $L_2$ norm) are quadratic forms and the optimization problems thus become relatively small semidefinite optimization problems.
In particular, if we do not consider the constraints (i.e., take $\kappa_\ell$ large enough), the $L_2$ approximation can be found numerically by the least squares approximation \cite[Sec. 11]{powell1981approximation}.


\begin{remark}
Other quantifications of the mismatch can of course also be used in the approximations \eqref{eq:optimizationBasic}--\eqref{eq:optimizationLeader}.
Moreover, there are several other properties that could also be of interest.
One of these is stability of the resulting moment system. Another is its invariance with respect to the set $\Ccal$. 
For the basic systems discussed in Sec.~\ref{sec:basicSystems}, this is the invariance of a linear system, which is related to positive and monotone systems \cite{angeli2003monotone, rantzer2015scalable}. A weaker condition related to positive systems are so-called eventually positive systems \cite{altafini2016minimal, altafini2015predictable}, which could also be of interest. For agents with interactions the approximate system is a quadratic system, 
whose invariance
is related to copositive matrices \cite{hall1963copositive}, \cite{dur2010copositive}.
\end{remark}

\subsection{Reconstruction of the distribution}\label{subsec:reco}
As mentioned in section \ref{subsec:inverse}, from a finite set of moments we can compute bounds on the distribution \cite{marzetta1984power} 
or
obtain a nominal estimate of the distribution \eqref{eq:moment_rep}.
This kind of reconstruction can in many cases be done by solving a convex optimization problem,
e.g., a problem on the form
\begin{subequations}\label{eq:reco_opt_prob_general}
\begin{align}
\min_{d\mu \in \mathcal{M}_+(K)} & \quad \int_K F(d\mu) \\
\text{subject to}  & \quad m \approx 
\int_K \phi(x)d\mu(x).\label{eq:reco_opt_prob_generalb}
\end{align}
\end{subequations}
Here, $F$ is a convex functional, and the constraints \eqref{eq:reco_opt_prob_generalb} are either enforcing exact matching of the moments or representing a suitable approximate matching
(see, e.g., \cite{borwein1991duality,
borwein1993partially, karlsson2016multidimensional, ringh2017multidimensional}). An example of such a problem for reconstructing a nominal estimate of the distribution  is the following total variation minimization problem with approximate moment matching
\begin{align}\label{eq:reco_opt_prob}
\min_{\Phi \geq 0, \varepsilon \geq 0} & \quad \int_K |\nabla \Phi(x)|d\sigma(x) + \lambda \, \varepsilon  \\
\text{subject to}  & \quad  | m_k \!-\! \int_K \phi_k(x)\Phi(x)d\sigma(x) | \leq \varepsilon, \, k = 1, \ldots, M, \nonumber
\end{align}
where $\lambda$ is a regularization parameter, and where we optimize over the set of absolutely continuous measures $d\mu(x)=\Phi(x)d\sigma(x)$ with respect to the Lebesgue measure $\sigma$.

In some cases a reconstruction of the full distribution may not be needed, e.g., when one needs to bound the number of agents in a hazardous region or ensure that all agents have reached a safe zone. A possible  formulation of \eqref{eq:reco_opt_prob_general} for such problem is to determine the maximal or minimal mass in a given subregion $\Omega\subset K$:
\begin{align}
\maxmin_{d\mu \in \mathcal{M}_+(K)} & \quad \int_\Omega d\mu  \label{eq:minmaxmass}\\
\text{subject to}  & \quad  m = \int_K \phi(x)d\mu(x).\nonumber
\end{align}
These are convex problems, and the resulting bounds are sometimes referred to as Cybenko bounds (cf. \cite{marzetta1984power, karlsson2013uncertainty}).


\begin{figure}[t]
  \centering
  \includegraphics[width=0.45\textwidth]{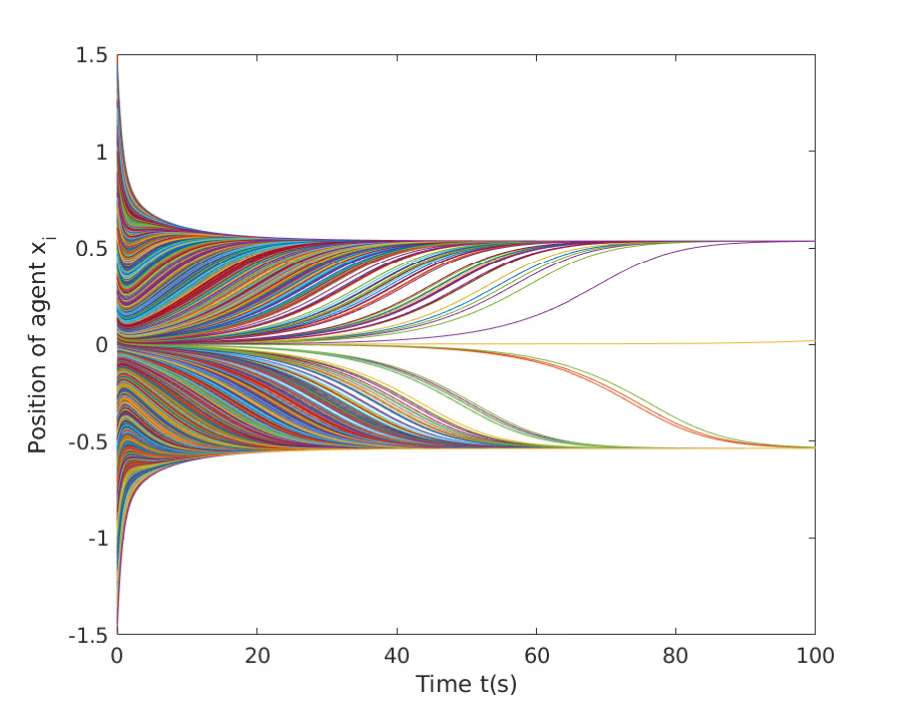}
  \vspace*{-10pt}
  \caption{The behavior of system \eqref{eq:Dynamics_Toy} with $10^4$  particles. The agents converge to a formation consisting of three clusters.}\label{Fig:Toy_behavior}
  \vspace{-15pt}
\end{figure}

\section{Numerical example}\label{sec:example}

In this section the collective behaviors of two particular multi-agent systems are investigated via the proposed approach.
In the first example, we consider a system with interacting agents.
For several different kernels the overall particle distribution can be captured using only a few moments to represent the system.
In the second example a pedestrian crowd moving in two-dimensional (2-D) space is modeled using the moment-based approach. The result shows that the method is applicable also for leader-follower scenarios.




\subsection{1-D example: A system governed by a spatial field and interactions}\label{subsec:1D_example}

We begin with a 1-dimensional scenario where the particles are driven by a time-invariant spatial field plus a repulsive influence between each pair of individuals. The dynamics of particle $i$ is
\begin{equation}\label{eq:Dynamics_Toy}
  \dot{x}_i = f(x_i) + \frac{1}{N} \sum_{j=1}^{N} g(x_i, x_j) \quad \mbox{ for } i=1, 2, \dots, N,
\end{equation}
where $x_i \in \mathbb{R}$ is the state of particle $i$, $f(x)=-x$ is a stabilizing vector field, and $g(x,y)=2e^{-0.6(x-y)^2}(x-y)$ is the repulsive interaction with exponentially decaying influence. Note that in this case, $f$ and $g$ satisfy the conditions in Remark~\ref{rem:whenKbounded} and therefore Assumption~\ref{ass:inK} holds for $T = \infty$. In fact, one can verify that for $\xi:=2/\sqrt{1.2}e^{-1/2}\approx 1.1$ the interval $[-\xi, \xi]$ is an invariant set for each system in \eqref{eq:Dynamics_Toy}.

The behavior of a system consisting of $10^4$  homogeneous particles\footnote{The initial position of each agent was drawn from a uniform distribution on the interval $[-1.5, 1.5]$} governed by dynamics \eqref{eq:Dynamics_Toy} is simulated, and the trajectories of all particles are shown in Figure~\ref{Fig:Toy_behavior}.
As can be seen, the collective behavior of the system gives rise to a formation consisting of three clusters. 

\begin{figure}[b]
  \centering
  \includegraphics[width=0.42\textwidth]{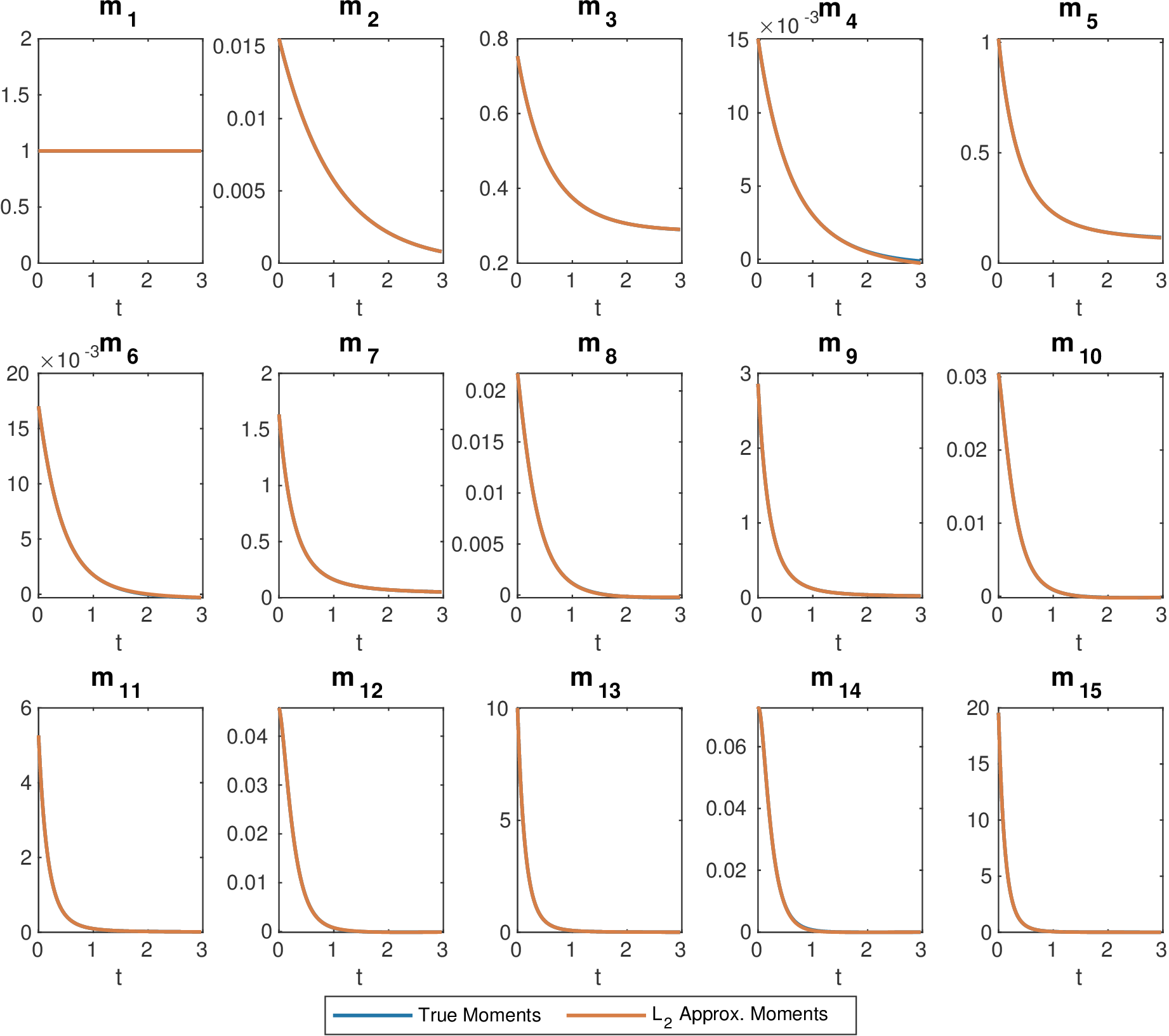}
  \caption{Moment trajectories for time  $t \in [0,3]$, using $15$ monomial kernel functions, i.e., $\phi_k(x)=x^{k-1}$, where $k=1, \dots, 15$. The trajectories are for both the true system and the $L_2$ approximated system. (As can be seen, the two trajectories overlap almost perfectly.)}\label{Fig:moment_traj_0to3s}
\end{figure}

\subsubsection{Modeling the collective behavior using moments}
Next we model the system using moments.
We consider the following three sets of kernel functions:
\begin{itemize}
\item Monomials, i.e., $\phi_k(x)=x^{k-1}$, where $k=1, \dots, 15$.
\item Chebyshev polynomials of the first kind, orthogonal on the interval $K=[-2,2]$, i.e.,
\begingroup
   \allowdisplaybreaks
\begin{align}
\phi_{1}(x)\!&=\!1, \label{eq:chebyshev} \\
\phi_{k}(x)\!&=\!(k\!\!-\!\!1) \!\!\sum_{i=0}^{k\!-\!1} \!\!\frac{(-\!2)^{i} \!(k\!\!+\!\!i\!\!-\!\!2)!}{(k\!\!-\!\!i\!\!-\!\!1)!(2i)!} \! \left(\!1\!-\!\frac{x}{2} \right)^i \mbox{for } k \!=\! 2, \!\ldots,\! 15. \nonumber
\end{align}
\item Monomials multiplied with Gaussian functions\footnote{This gives more spatial localization of the information carried by each moment, but in order to still convey some global information of the distribution we take the last kernel function to be the constant function, i.e., $\phi_{15}(x) \equiv 1$. }
\begin{align}
\phi_{7i+j}(x)&=\tfrac{x^i\exp\left(\!-\frac{(x - \rho_j)^2}{\sigma^2}\right)}{\sqrt{2\pi \sigma^2}}  \mbox{ for } i = 0,1;\, j = 1, \ldots, 7,\nonumber\\
\phi_{15}(x)&= 1,\label{eq:gaussian}
\end{align}
where the centers $\rho_j\in \mR$ are taken as equidistant points within interval $[-1.5, 1.5]$ while $\sigma$ is set to $2/3$.
\endgroup
\end{itemize}
In each set a total of 15 kernel functions are used, and for each set the region on which we approximate the dynamics is taken to be the closed interval $K=[-2,2]$.
The approximations are carried out for minimizing the errors in \eqref{eq:errorBasic} and 
\eqref{eq:errorInter} using the formulations \eqref{eq:optimizationBasic} and \eqref{eq:optimizationInter}
with $f(x)=-x$ and $g(x,y)=2e^{-0.6(x-y)^2}(x-y)$, respectively.
For all three choices of kernel functions we attempt to compute the approximations using both the $L_2$ norm and the $L_\infty$ norm. These problems are solved without bounds on the logarithmic norms (i.e., $\kappa_\ell=\infty$ for all $\ell$). In addition,  we also consider the model reduction with logarithmic norm constraints for the case with Gaussian kernel functions \eqref{eq:gaussian}.
Moreover, in order to see how well the distribution of agents can be recovered we  solve the reconstruction problem \eqref{eq:reco_opt_prob} for $t= 3$ and $t = 100$.
The optimization problems are solved in Matlab using CVX \cite{grant2008graph, cvx} (except for the $L_2$ approximations without bounds, which can be computed by solving linear systems).

\begin{figure}[tb]
 \centering
  \includegraphics[width=0.42\textwidth]{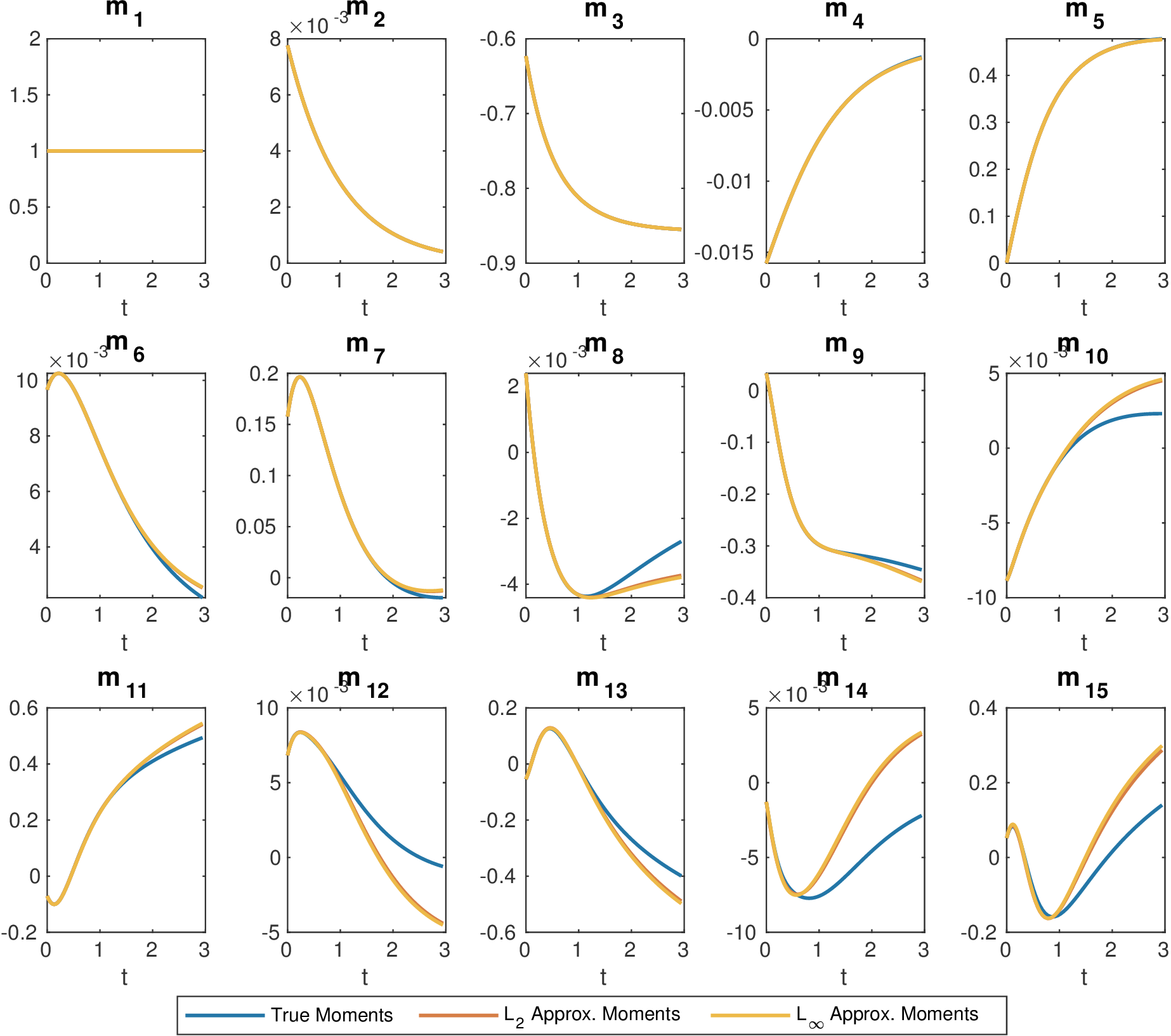}
  \caption{Moment trajectories for time  $t \in [0,3]$, using $15$ Chebyshev polynomials \eqref{eq:chebyshev} as kernel functions. The trajectories are for the true system, the $L_2$ approximated system, and the $L_\infty$ approximated system.}\label{Fig:moment_traj_cheb_t0to3}
\end{figure}

\begin{figure}[tb]
 \centering
  \includegraphics[width=0.42\textwidth]{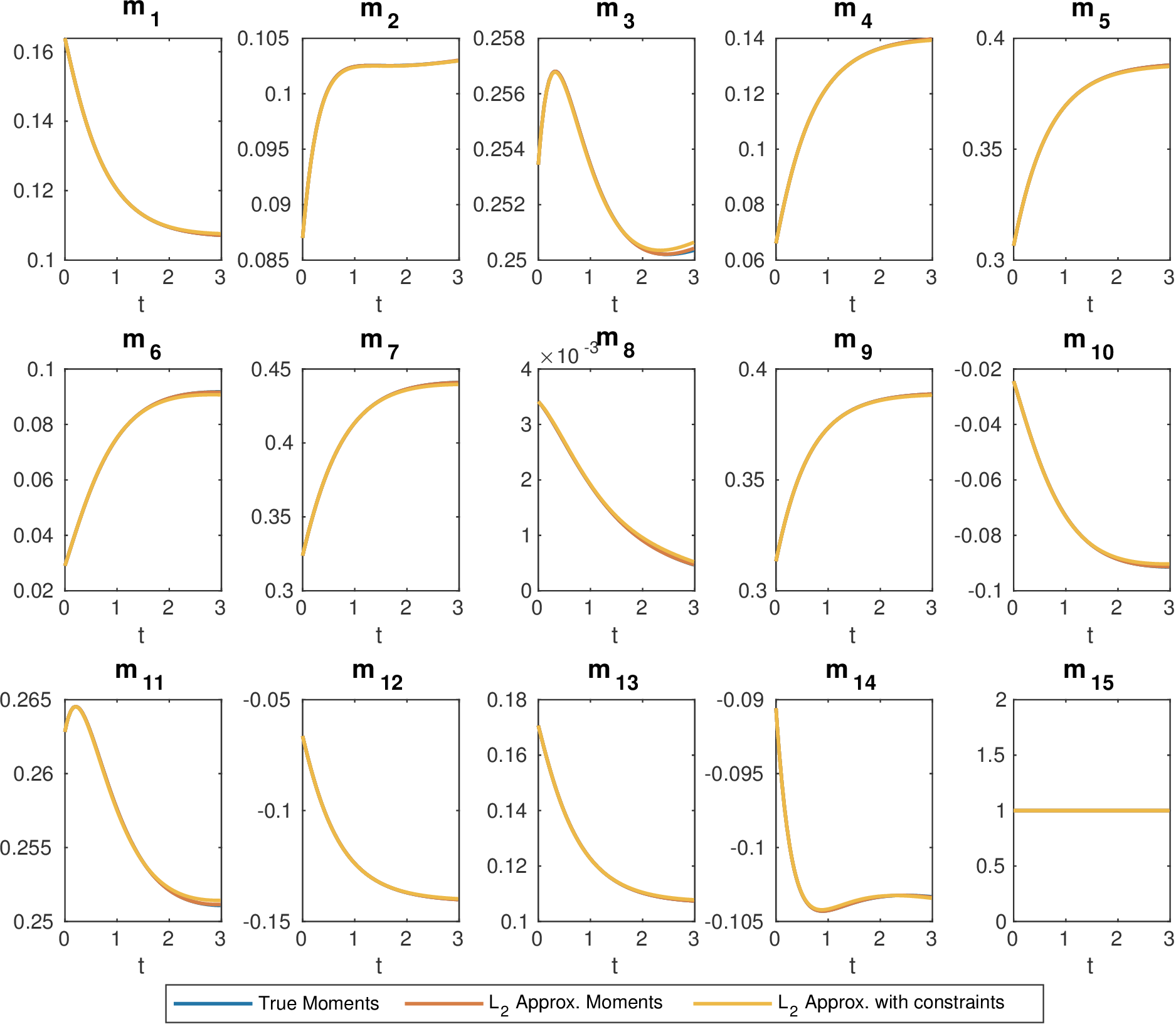}
  \caption{Moment trajectories for time $t \in [0,3]$, using  Gaussian kernel functions \eqref{eq:gaussian}.
  The trajectories are for the true system, and the $L_2$ approximated system with and without constraints on the logarithmic norms.}\label{Fig:moment_traj_gauss_poly_t0to3}
\vspace{-10pt}
\end{figure}

\subsubsection{Simulation results and discussion}

The trajectories of the true and approximate moments are compared on the interval $t \in [0,3]$, and shown for the monomials in Figure~\ref{Fig:moment_traj_0to3s}, for Chebyshev polynomials \eqref{eq:chebyshev} in Figure~\ref{Fig:moment_traj_cheb_t0to3}, and for Gaussian functions \eqref{eq:gaussian} in Figure~\ref{Fig:moment_traj_gauss_poly_t0to3}.
For the cases with monomial and Gaussian kernels the optimization solver did not converge for $L_\infty$ approximation, and those results are therefore omitted.

In both the approximate models with monomials and Gaussian kernels, all the approximate moments match the true moments well in the interval $[0,3]$.
For the case with Chebyshev polynomials, some of the higher order moments start to deviate from the true moments at around time 1, whereas the errors in the lower order moments remain small throughout the whole interval $[0,3]$. Next we consider the approximate models based on $L_2$ error using Gaussian kernels with and without bounds on the logarithmic norms of the system for a longer time interval.
As shown in Figure~\ref{Fig:moment_traj_gauss_poly_t0to3} the two approximations behave well for $t \in [0, 3]$, but as shown in
Figure~\ref{Fig:moment_traj_gauss_poly_3to100s} the approximate moments have an oscillating trajectory for $t\in [3, 100]$ for the case without the logarithmic norm bound. On the other hand, the moment trajectory corresponding to the model with logarithmic norm bound has a stable behavior, which illustrates that the use of logarithmic norm bounds can have a stabilizing effect on the moment dynamics.

\begin{figure}[t]
 \centering
 \makebox[0.49\textwidth][c]{\includegraphics[width=0.5\textwidth, trim={1.2cm 0.2cm 0 0.5cm},clip]{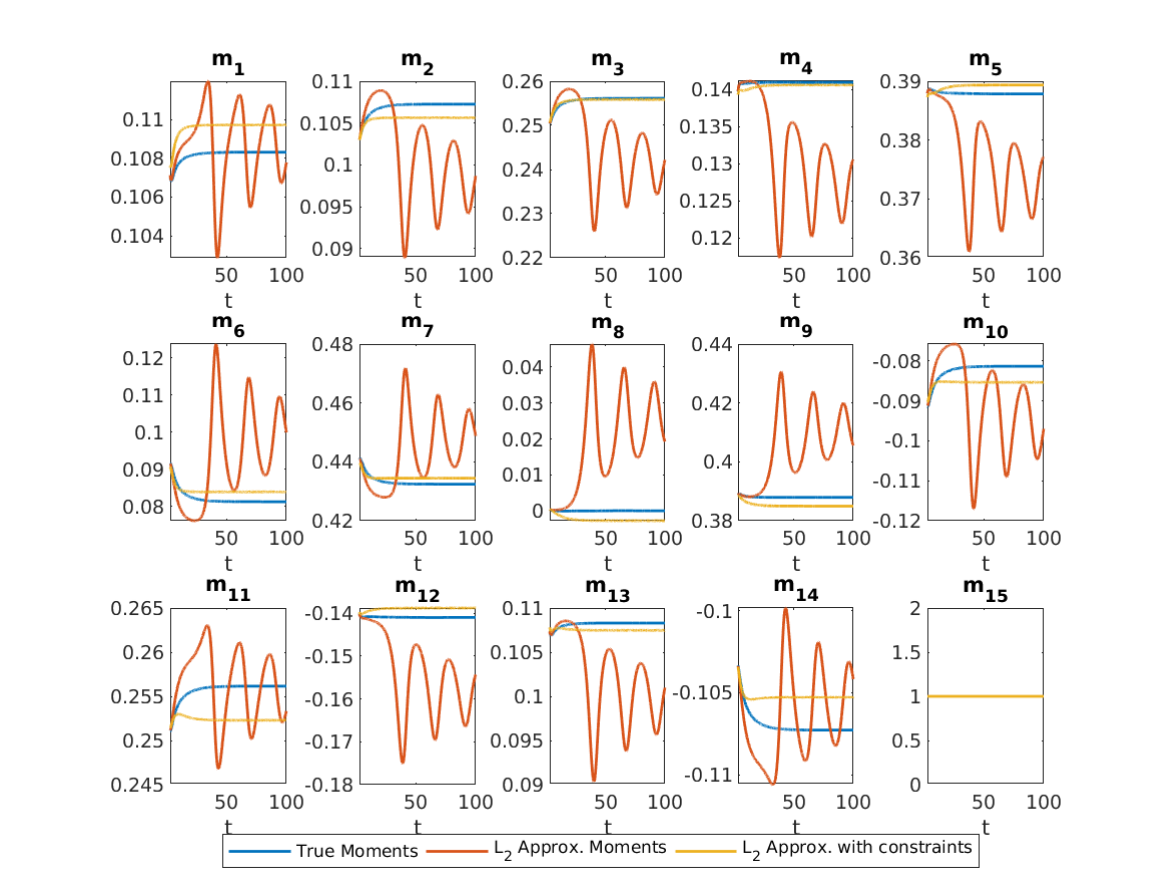}}
  \caption{Moment trajectories for time $t \in [3,100]$, using  Gaussian kernel functions \eqref{eq:gaussian}.
 Continuation of the moment trajectories from Figure~\ref{Fig:moment_traj_gauss_poly_t0to3}.
  }\label{Fig:moment_traj_gauss_poly_3to100s}
  \vspace*{-10pt}
\end{figure}

However, what is important for understanding the collective behavior of the underlying system is the information carried about the distribution of particles by the approximate moments.
In Figure~\ref{Fig:reco} we present total variation reconstructions \eqref{eq:reco_opt_prob} performed at times $t= 3$ and $t = 100$. 
The figure shows reconstructions from the true and approximate moments,
using the monomials as kernel functions, as well as  histograms of the true particle distributions.
From the results in Figure~\ref{Fig:reco} we see that the approximate moments capture the behavior of the overall system quite well
and that the difference in the true and approximate moments only gives rise to a small difference between the reconstructed distributions. Reconstructions from true and approximate moments at the time points $t=3$ and $t=100$ show decent results also for Chebyshev and Gaussian kernel functions, but are omitted due to space considerations. It is somewhat surprising, in particular for the Gaussian kernels, that the reconstruction is good also for time $t=100$, since there are errors in the approximate moments as seen in Figure~\ref{Fig:moment_traj_gauss_poly_3to100s}. This could possibly be explained by the fact that the approximate dynamics manages to capture the correct steady state of the true system. Note that these are only nominal reconstructions and that a more thorough analysis needs to be performed in order to determine, e.g., bounds on the number of agents in a certain region (cf. \eqref{eq:minmaxmass}). This will be subject to further research.

\begin{figure}[t]
 \centering
 \makebox[0.49\textwidth][c]{\includegraphics[width=0.5\textwidth, trim={1cm 0.2cm 0 0.9cm},clip]{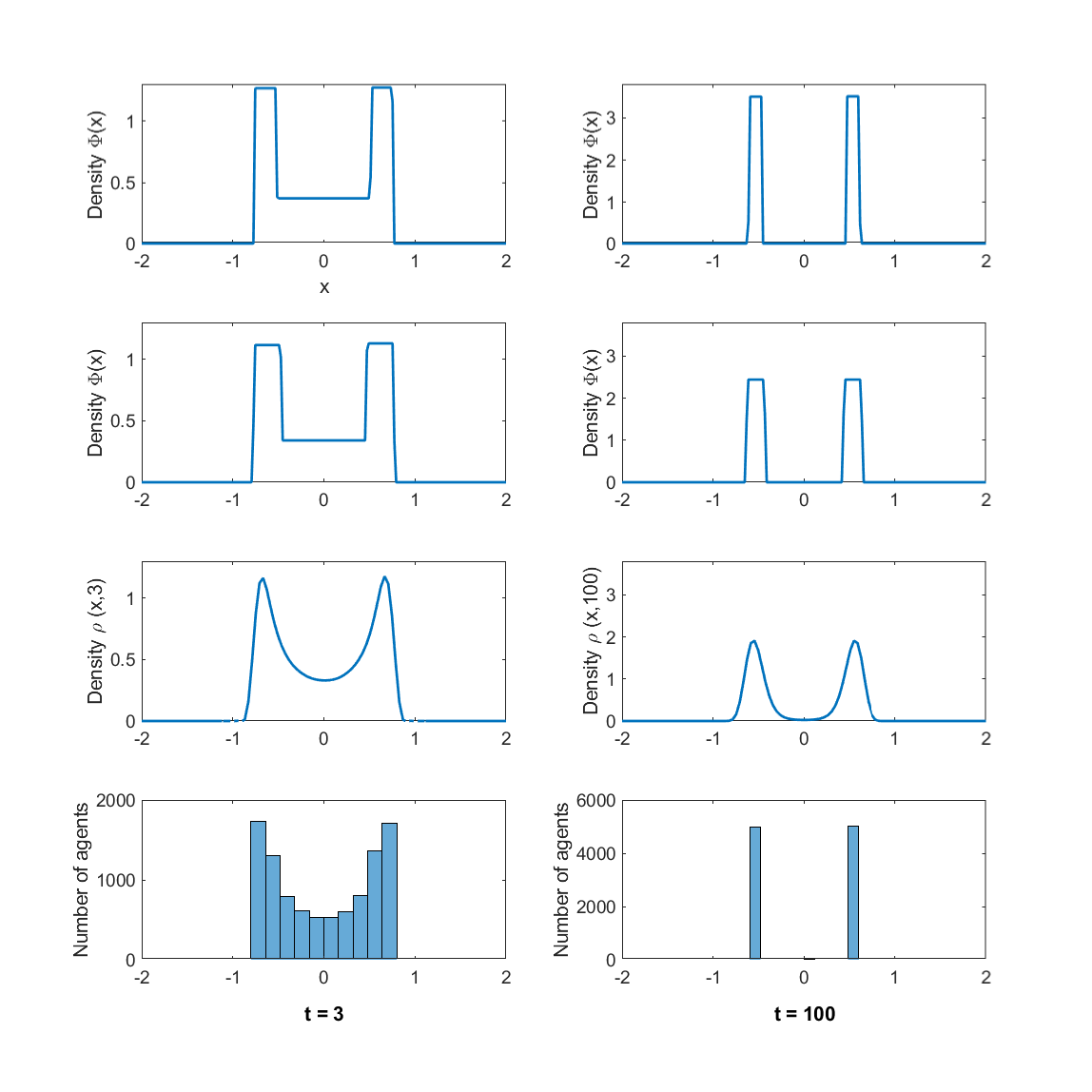}}
 \vspace{-25pt}
\caption{
Estimated and true distribution of agents at times $t=3$ and $t=100$.
The first two rows show reconstructions of the distribution density $\Phi(x)$ from monomial moments using \eqref{eq:reco_opt_prob}, where the first row is using true moments and the second row is using approximate moments obtained via $L_2$ approximation of the dynamics.
The third row is the density function $\rho(x,t)$ obtained by solving the Vlasov equation \eqref{eq:McKean_Vlasov_eqn}, and the bottom row is a histogram representation of the true distribution.
}\label{Fig:reco}
\end{figure}

\subsubsection{Solving by mean-field method}
Another method that can be used for computing the macroscopic evolution of \eqref{eq:Dynamics_Toy} is the mean-field approach, in which the occupation measure $d\mu_t$ in \eqref{eq:moment_rep} is approximated\footnote{The occupation measure $d\mu_t$ converges weakly to $\rho(\cdot, t)dx$ as the number of particles $N\to \infty$.}
by the density $\rho(x,t)$.
The density function $\rho$ is governed by the \textit{Vlasov} equation \cite{dobrushin1979vlasov, jabin2017mean}:
\begin{equation}\label{eq:McKean_Vlasov_eqn}
\frac{\partial}{\partial t} \rho(x,t) \!=\! -\nabla_x \cdot \Big[ \rho(x,t) \Big(f(x) +\! \int_K \!\! g(x,y) \rho(y, t) dy \Big)\!\Big],
\end{equation}
where $f$ is the spatial vector field and $g$ is the interaction law, corresponding to the dynamics \eqref{eq:Dynamics_Toy}.
The partial differential equation \eqref{eq:McKean_Vlasov_eqn} is a conservation law without a diffusion term. A popular approach for solving such PDEs is the Lax-Friedrichs method \cite{leveque2002finite}, which by introducing artificial viscosity manages to maintain numerical stability. 

We implement the Lax-Friedrichs method to solve \eqref{eq:McKean_Vlasov_eqn} numerically, where the initial density $\rho(x,0)$ is chosen as the histogram%
\footnote{We use the histogram on $K=[-2,2]$ with 200 intervals.}
generated from the initial particle positions $\{x_i(0)\}_{i=1}^N$.
The resulting density $\rho(\cdot,t)$ at times $t=3$ and $t=100$ is shown in the third row of Figure~\ref{Fig:reco}. From the figure we see that the obtained estimate agrees well with the true distribution.
\footnote{When solving the PDE model, the artificial viscosity term results in a biased solution and the density function is considerably smoothed, as can be seen in Figure~\ref{Fig:reco}. To reduce the artificial viscosity while guaranteeing stability of the method both the time grid and the space grid need to be refined, resulting in increased computational complexity.}

To compare the accuracy to two methods, Figure~\ref{Fig:Wasserstein-1 error} shows the $L_1$ distance between the distribution functions,%
\footnote{That is, we use the distance $W_1(d\mu_0,d\mu_1)=\int|\int_{-\infty}^{x}(d\mu_0-d\mu_1)|dx.$}
which in the one-dimensional case ($d=1$) coincides with the Wasserstein-1 distance \cite[Thm. 2.9]{bobkovone}, between the true distribution of agents and the two approximate solutions obtained using the moment-based method or by solving \eqref{eq:McKean_Vlasov_eqn}.
The figure shows that solving \eqref{eq:McKean_Vlasov_eqn} gives a more accurate estimate in the beginning, while the moment-base method gives a more accurate estimate for larger times. This may be caused by the artificial viscosity term introduced in the numerical method.


\begin{figure}[tb]
 \centering
  \includegraphics[width=0.45\textwidth]{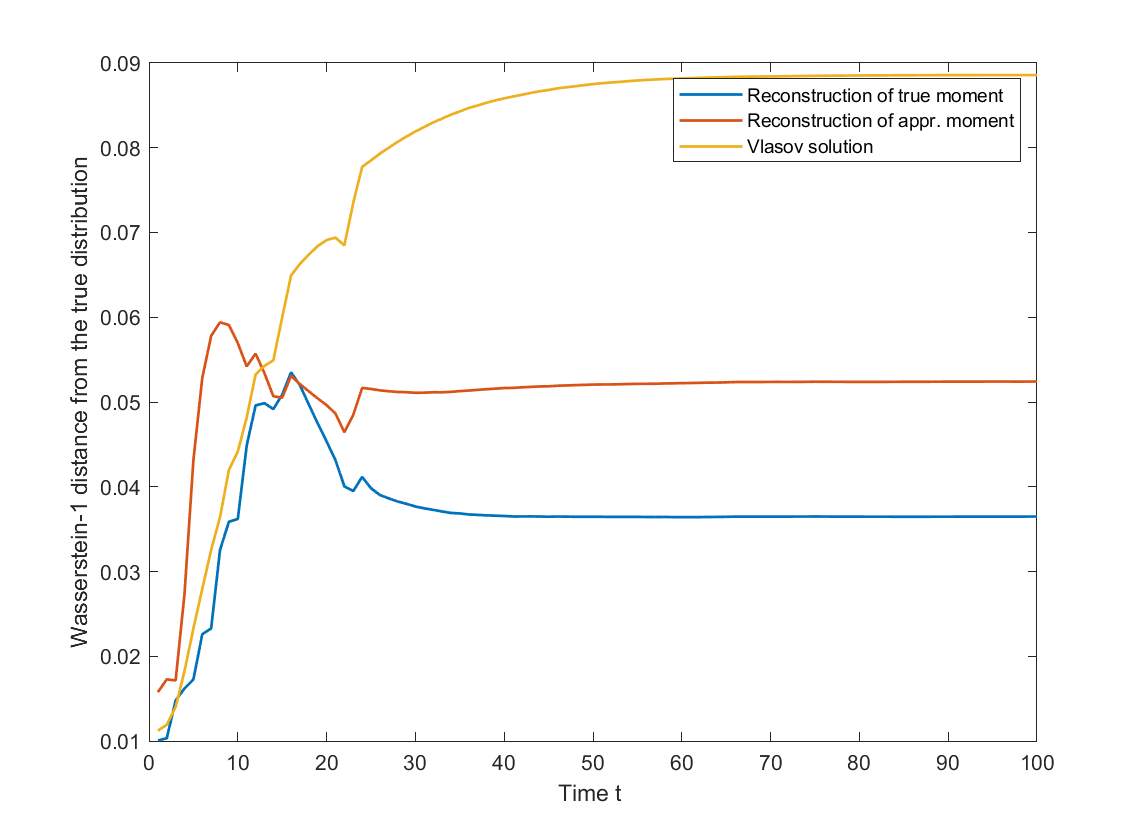}
  \caption{$L_1$ distance between the true distribution function of agents and the estimates obtained as described in Figure~\ref{Fig:reco}. 
}\label{Fig:Wasserstein-1 error}
\end{figure}

Comparing the complexity of the two methods: in the mean-field methods it is easy to obtain the models (kinetic PDEs) given the individual dynamics, but solving the resulting PDE is computationally expensive. Instead, in the moment-based method, the main computational effort is in computing the model for the moment dynamics, while solving the obtained ODE model is easier.
Moreover, the latter method decouples the discretization in time and space, in contrast to the former.
In particular, let $N_x$ and $N_t$ be the numbers of grid points in space and time, respectively. The computational complexity of solving the Vlasov equation is $\mathcal{O}(N_x^2  N_t)$.
To maintain numerical stability with a small viscosity term, the CFL condition \cite{leveque2002finite} requires the number of grid points $N_t$ to be at least $\mathcal{O}(N_x)$.
In the moment method, the complexity of solving the obtained ODE is $\mathcal{O}(N_t  M^3)$, and the complexity of obtaining a moment model (using $L_2$ approximation) is $\mathcal{O}(N_x M^2 + M^7)$, where $M$ is the number of kernels. As can be seen from this, the moment approach decouples the dependence of time and space into two separate procedures. 

\subsection{An example with leaders: Pedestrian dynamics}
In this subsection the collective behavior of a pedestrian crowd is modeled and simulated via the proposed method. It is shown that the moment-based modeling has practical potential and can reduce the computational complexity in applications of crowd simulation and crowd control.

We consider the pedestrian crowd with the leader-follower structure \cite{helbing2009pedestrian, yang2015shaping}, in which each general public acts as a follower and the individuals in charge of guiding the public as leaders.  Since the leaders have better knowledge of the whole environment, they are able to lead the crowd to reach certain goals, for example in an evacuation scenario the rescue workers (leaders) are sent to guide people (followers) escaping from a certain region in the safest and most efficient way.
We denote the positions of followers and leaders by  $x_i, y_j \in K \subset \mR^2$, respectively, where $i=1, \dots, N$ and $j=1, \dots, N_L$. Each follower is governed by the dynamics \cite{yang2015shaping}
\begin{equation*}
\dot x_i = \frac{1}{N} \sum_{j=1}^{N} g(x_i,x_j)+\sum_{j=1}^{N_L} \eta(x_i,y_j),
\end{equation*}
with the velocities specifying the interactions and leader-follower dynamics defined as
\begin{align}\label{eq:simu:Dyn_crowd}
g(x,y)&=\frac{4.8}{\|x-y\|+0.1} e^{-\frac{2\|x-y\|}{5}} (x-y),\\
\eta(x,y)&=\left(0.09+6 e^{-\frac{\|x-y\|}{50}} -\frac{6}{\|x-y\|+0.1}\right)(y-x). \nonumber
\end{align}

A pedestrian crowd with $N=10^3$ followers and $N_L=4$ leaders is simulated within a compact region $K=[-2,2]\times [-2,2]$. To focus the current work on modeling, we directly assign the trajectories of the leaders traversing the region $K$ with some sinusoidal detours as shown by the solid lines in Figure~\ref{Fig:Crowd_behavior}. The initial positions of the followers are drawn from a uniform distribution in the region $K_0=[-1.5,1.5]\times [-1.5,1.5]$,
which can be seen from the first snapshot in Figure~\ref{Fig:Crowd_behavior}. The other snapshots illustrate the motion of the crowd and it can be seen that the crowd formation can evolve in rather intricate patterns when being guided by the 4 leaders.

 \begin{figure}[tb]
  \makebox[0.49\textwidth][c]{\includegraphics[width=0.5\textwidth, trim={0.6cm 0 0 0},clip]{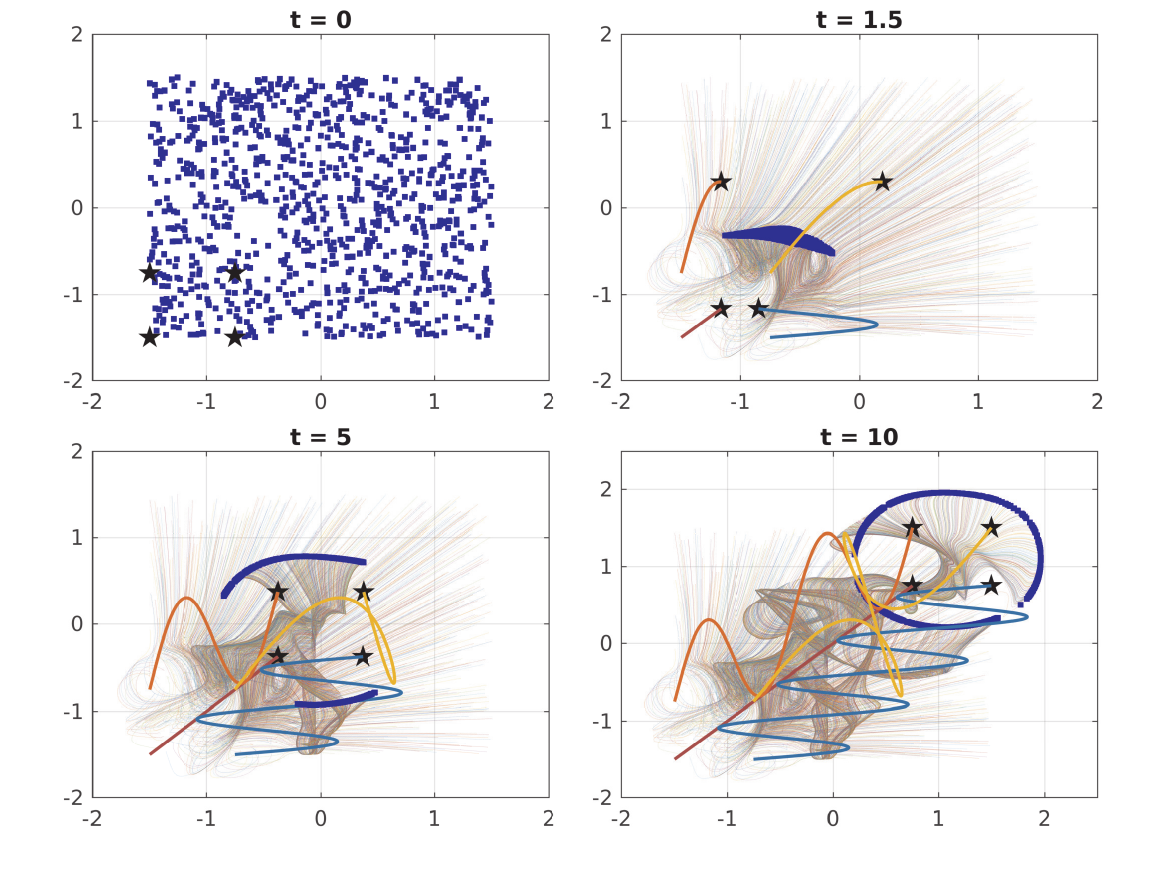}}
  \vspace{-20pt}
  \caption{ Trajectory snapshots of a pedestrian crowd with $10^3$ followers and four leaders at four time instants. Black pentagrams represent leaders and blue squares represent followers. The trajectories of leaders and followers are depicted by solid line and dashed line, respectively.}\label{Fig:Crowd_behavior}
\end{figure}

In this example, we use the polynomial kernels
\[
\phi_{k,\ell}(x,y)=\psi_{k,\ell}(x,y)=x^ky^\ell \quad \mbox{ for } (x,y) \in K
\]
where $k, \ell$ are positive integers satisfying $0 \leq k+\ell \leq 7$.
According to the proposed approach, the $L_2$ approximations are carried out for minimizing \eqref{eq:errorInter} and \eqref{eq:ApproSysLeader1} with functions $g(x,y)$ and $\eta(x,y)$ in \eqref{eq:simu:Dyn_crowd}. The trajectories of the approximate moment system are compared with those of the real moment system in Figure~\ref{Fig:crowd_moment_traj}. The simulation result shows that the proposed method predicts the moments quite well under this scenario with intricate leader interactions.

\begin{figure}[tb]
\hspace{-10pt}
 \makebox[0.5\textwidth][c]{\includegraphics[width=0.49\textwidth, trim={2cm 0cm 1.7cm 1cm}, clip]{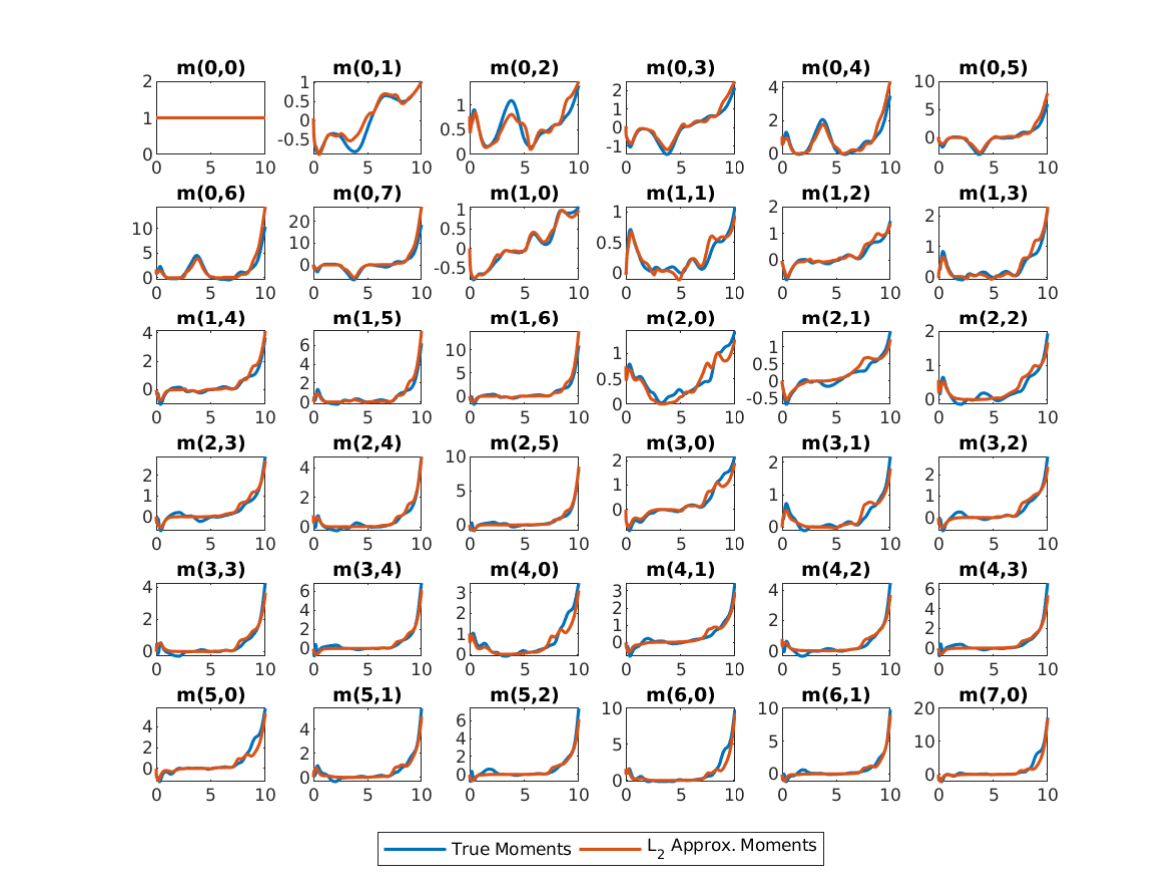}}
  \caption{Moment trajectories for time (abscissa-axis)  $t \in [0,10]$, using polynomial kernel functions, where subfigure ${m}(k,\ell)$ is for the moment corresponding to kernel $\phi_{k,\ell}(x,y)=x^ky^\ell$. The trajectories are for both the true system and the $L_2$ approximated system.}\label{Fig:crowd_moment_traj}
\end{figure}

Moreover, total variation reconstructions \eqref{eq:reco_opt_prob} are also performed to recover the distributions from the moments at times $t= 1.5$ and $t = 10$. Figure~\ref{Fig:crowd_reco} shows the reconstructions from the true and approximate moments, as well as histograms from the true crowd distribution. These results show that although the order of the system is reduced dramatically by considering the moment system instead of individual dynamics, much of the information related to position and formation of the crowd is captured accurately by the moment dynamics.

\begin{figure}[tb]
  \makebox[0.49\textwidth][c]{\includegraphics[width=0.5\textwidth, trim={1.2cm 0.2cm 0 0.6cm},clip]{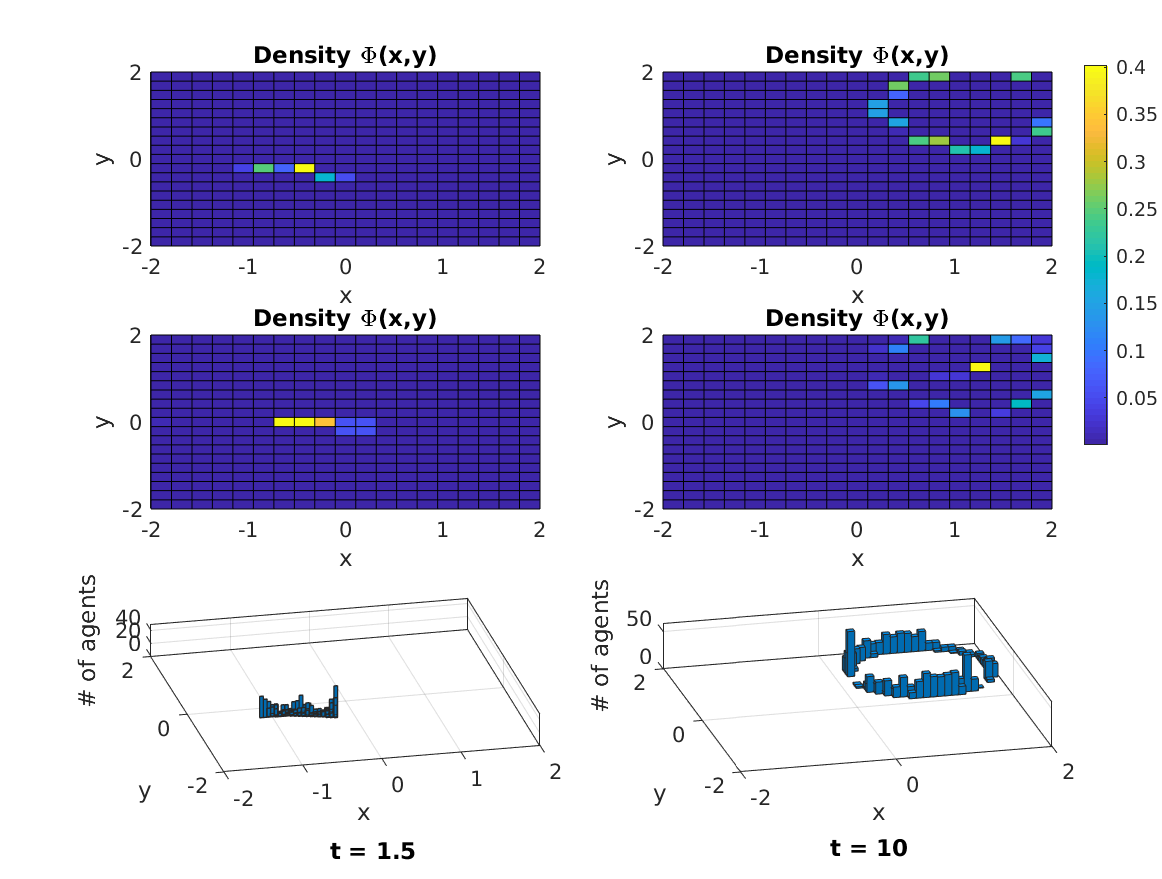}}
  \caption{Moment reconstructions for the crowd system using \eqref{eq:reco_opt_prob}, and histogram of the true pedestrain distribution. The top row is a reconstruction of distribution density $\Phi(x,y)$} using the true moments and the middle row is using the moments with dynamics obtained via $L_2$ approximation.\label{Fig:crowd_reco}
\end{figure}

\section{Conclusions and future directions}\label{sec:conclusion}
This paper introduces a moment-based approach for modeling the collective behavior of a multi-agent system, without having to track each individual agent.
This is done by approximating the dynamics of generalized moments with linear or quadratic
systems,
and the approximate system is computed via a convex optimization problem where trade-offs between the accuracy and stability of the moment dynamics can be controlled.

As a final remark, we note that the approach can be generalized, in a straightforward manner, to systems with control. For example, the setup in Remark~\ref{rkm:controlProblem}, as well as the case when the agents are governed by
$\dot x_i(t)=f(x_i(t)) + Gu(t),$ for $i = 1, \ldots, N.$
The latter can be approximated by the bilinear system
\begin{equation*}
\frac{d}{d t}
\!
\begin{bmatrix}
m_1 \\ \vdots \\ m_M
\end{bmatrix}
\!\!\approx \!\!
\begin{bmatrix}
a^1_{1} & \cdots & a^1_{M} \\
\vdots  &        & \vdots  \\
a^M_{1} & \cdots & a^M_{M}
\end{bmatrix}
\! \!
\begin{bmatrix}
m_1 \\ \vdots \\ m_M
\end{bmatrix}
 \!+ \!
\begin{bmatrix}
m^T \Lambda_1^T G \\ \vdots \\ m^T \Lambda_M^T G
\end{bmatrix}
u,
\end{equation*}
where $\Lambda_k = [\lambda^k_{j, \ell}]_{j, \ell = 1}^M$, in a similar manner as described in Section~\ref{sec:basicSystems}.
In this paper we focus on the modeling, and we leave the control aspects to be further studied.

There are several open questions that would strengthen the proposed framework. One is how to optimally select the kernel functions in order to guarantee an accurate moment dynamics while also ensuring that the inverse problem provides an accurate reconstruction of the distribution. This may depend on the specific macro-scale properties of the reconstruction that one wants to recover. Another open problem is an asymptotic theory that shows that the reconstruction error converges to zero (in a suitable sense) as the number of moments goes to infinity. A third problem of interest is how to design a control for the proposed moment-based model in order to, e.g., steer the behavior towards a desired outcome.



\vspace{-1em}
\section*{Acknowledgemewnts}
The authors would like to thank Patrick Henning for the valuable input on solving the Vlasov equations.

\appendix

\allowdisplaybreaks


\subsection{Proof of Theorem \ref{thm:ErrorBoundBasic}}\label{sec:app1}

  Let $E_k(t)=\int_{x\in K} \varepsilon_k (x) d\mu_t(x) $, with $\varepsilon_k (x)$ defined in \eqref{eq:errorBasic}, then the moment system (\ref{eq:MomentSysBasic})  can be rewritten as
  \[
    \dot m_k(t) = \sum_{\ell = 1}^M a^k_{\ell} m_\ell(t) + E_k(t).
  \]
  Since by \eqref{eq:ApproSysBasic} the approximate system  is $\dot{\overline{m}}_k = \sum_{\ell = 1}^M a^k_{\ell}\overline{m}_\ell$, the error dynamics $\Delta m_k = m_k -\overline{m}_k$ satisfies
  \begin{equation}\label{eq:pf:ErrorSystem}
    \Delta \dot m(t) = A \Delta m(t) +E(t),
  \end{equation}
   where $E(t):=[E_1(t), \dots, E_m(t)]^T$. By solving the linear system of ODEs \eqref{eq:pf:ErrorSystem}, using Lemma~\ref{lem:expBound} and the fact that the total mass of the distribution \eqref{eq:moment_rep} is $1$, the error of the moments can be bounded by
   \begingroup
   \allowdisplaybreaks
   \begin{align*}
     & \|\Delta m(t)\| \leq \|\Delta m(0)\| \,\|e^{At}\|  +\|\int_{0}^{t} e^{A(t-s)} E(s) ds\|   \\
     & \leq \|\Delta m(0)\| \, e^{\nu[A]t}  +\int_{0}^{t} e^{\nu[A](t-s)} \|E(s)\| ds \\
     & \leq \|\Delta m(0)\| \, e^{\nu[A]t}  +\int_{0}^{t} e^{\nu[A](t-s)} \sqrt{\sum_{k} \max_{x\in K} \varepsilon_k (x)^2} ds.
   \end{align*}
   \endgroup
   By computing the last integral, the assertion follows.\hfill$\square$

\subsection{Proof of Theorem \ref{thm:ErrorBoundInter}}\label{sec:app2}
\begin{figure*}[t]
\footnotesize
\begin{align}
      & D_t^+ \|m(t)-\overline{m}(t)\| = \limsup_{h \to 0^+} \frac{\|m(t+h)-\overline{m}(t+h)\|-\|m(t)-\overline{m}(t)\|}{h} \nonumber \\
      & = \limsup_{h \to 0^+} \frac{\|m(t) + h \left( \mB(m(t)) + E(t) \right) + \ordo(h^2) - \big(\overline{m}(t) + h \mB(\overline{m}(t)) + \ordo(h^2) \big)\|-\|m(t)-\overline{m}(t)\|}{h} \nonumber \\
      & \leq  \limsup_{h \to 0^+} \frac{\|m+h \mB(m)- \big(\overline{m} + h \mB(\overline{m})  \big)\|-\|m-\overline{m}\|}{h} + \|E(t)\| \nonumber 
      =  \limsup_{h \to 0^+} \frac{\| [I + h\mB] (m)  - [I + h\mB](\overline{m})\|-\|m-\overline{m}\|}{h} + \|E(t)\| \nonumber \\
      & \leq \limsup_{h \to 0^+} \frac{ L[I+h\mB] \, \|m-\overline{m}\|-\|m-\overline{m}\|}{h} +\|E(t)\| 
      = M[\mB] \|m-\overline{m}\| +\|E(t)\| \label{eq:long_ineq}.
\end{align}
\vspace{-10pt}
\hrulefill
\vspace*{-10pt}
\end{figure*}

By \eqref{eq:errorInter} the dynamics \eqref{eq:true_dyn} can be rewritten as
\begin{equation*}
  \dot m_k(t)= m(t)^T B_k m(t) + E_k(t),
\end{equation*}
where $E_k(t)= \int_{x\in K} \int_{y\in K} \varepsilon_k  d\mu_t(x)d\mu_t(y)$. By introducing the notation $\mB(m):=(m^T B_1^Tm, \dots, m^T B_M^Tm)$ and $E(t)=(E_1(t), \dots, E_M(t))$, we write the dynamics in vector form as $\dot m(t)= \mB(m) + E(t)$.
Similarly to the proof of Theorem~\ref{thm:ErrorBoundBasic}, we now investigate the dynamics of the norm of the error: $\| \Delta m(t) \| = \|m(t)-\overline{m}(t)\|$. To this end, we consider the
Dini derivative $D_t^+$ (also known as the upper (right-hand) derivative \cite[pp. 110-111]{royden2010real}) 
of $\|m(t)-\overline{m}(t)\|$. 
This gives the inequalities shown in \eqref{eq:long_ineq}, where the second equality is a Taylor series expansion, the first inequality is the triangle inequality (and  note that $\lim_{h\to 0}\ordo(h^2)/h=0$), the second inequality follows by \eqref{Def:L}, and the last equality follows by \eqref{eq:DefNonLN} since the limit exists.

Moreover, by Lemma~\ref{lem:LNormforC1} we have
\begingroup
\allowdisplaybreaks
\begin{align*}
M[\mB] &= \sup_{m \in D} \nu[ \nabla \mB(m)] =\sup_{m \in D} \nu\left[
	\begin{pmatrix}
	m^T(B_1+B_1^T)\\
	\vdots\\
	m^T(B_M+B_M^T)
	\end{pmatrix}
	\right] \\
&=\sup_{m \in D} \nu\left[
	\sum_{\ell=1}^M m_\ell
	\begin{pmatrix}
	b_{\ell, 1}^1+b_{1,\ell}^1&\cdots&b_{\ell, M}^1+b_{M,\ell}^1\\
	\vdots&&\vdots\\
	b_{\ell, 1}^M+b_{1,\ell}^M&\cdots&b_{\ell, M}^M+b_{M,\ell}^M
	\end{pmatrix}
	\right]\\
& \leq \sup_{m \in D} \sum_{\ell=1}^M \nu\left[ m_\ell \widetilde{B}_{\ell}	\right]=: \beta.
\end{align*}
\endgroup
Combining this with \eqref{eq:long_ineq}, it follows that
\[
D_t^+ \| \Delta m(t) \| := D_t^+ \|m(t)-\overline{m}(t)\| \leq \|E(t)\| + \Bbound \|\Delta m(t)\|,
\]
and integrating the above inequality gives that
\[
\|\Delta m(t)\| - \|\Delta m(0)\| \leq \int_0^t \| E(s) \| ds + \int_0^t \Bbound \|\Delta m(s)\| ds.
\]
If $\Bbound = 0$ the conclusion follows since the distribution \eqref{eq:moment_rep} has total mass $1$. If $\Bbound \neq 0$, using the Gr\"onwall-Bellman inequality
(see, e.g., \cite{khalil1992nonlinear}) we get that
\begin{align*}
& \|\Delta m(t)\| \leq  \|\Delta m(0)\| + \int_0^t \| E(s) \| ds \\
& \phantom{xxx} + \int_0^t \left( \|\Delta m(0)\| + \int_0^s \| E(\tau) \| d\tau \right) \Bbound e^{\Bbound(t-s)} ds \\
& \phantom{xxx} \leq \|\Delta m(0)\| + t \, \sqrt{\sum_{k=1}^M \max_{x,y\in K} \varepsilon_k (x,y)^2} \\
& \phantom{xxx} + \! \int_0^t \! \! \left( \! \! \|\Delta m(0)\| + s   \sqrt{ \sum_{k=1}^M \max_{x,y\in K} \varepsilon_k (x,y)^2} \right) \! \! \Bbound e^{\Bbound(t-s)} ds
\end{align*}
and the result follows by straight-forward integration. \hfill$\square$

\subsection{Proof of Lemma~\ref{thm:P_k,nMontonicity}}\label{sec:Apx_Pf_montonicity}

We start with a lemma presenting some properties of the $L_\infty$ approximation error $\mathbf{E}_n(x^k)$.
Note that in the entire Appendix~\ref{sec:Apx_Pf_montonicity}, for ease of notation, by $\| \cdot \|_{\infty}$ we denote $\| \cdot \|_{\Linf{[-1,1]} }$.

\begin{lemma}\label{thm:PropertiesEn_x^k}
  The $L_\infty$ approximation error of $x^k$ by polynomials $\mathcal{P}_n[-1,1]$, denoted $\mathbf{E}_n(x^k)$, has the following properties,
  \begin{enumerate}[(a)]
    \item $\mathbf{E}_n(x^k) \geq \mathbf{E}_{n+1}(x^k)$;
    \item $\mathbf{E}_n(x^k) \geq \mathbf{E}_{n+1}(x^{k+1})$;
    \item If $n$ and $k$ are both even, $\mathbf{E}_n(x^k)=\mathbf{E}_{n+1}(x^k)$;
    \item If $n+k$ is odd, $\mathbf{E}_n(x^k) \geq \mathbf{E}_{n}(x^{k+1})$.
  \end{enumerate}
\end{lemma}

\begin{proof}
  Note that the $L_\infty$ approximation of any continuous function by elements in $\mathcal{P}_n$ is unique (cf. \cite[Thm.~7.6]{powell1981approximation}).
  In the following proof, we denote $p_n^*\in \mathcal{P}_n$ and $p_{n+1}^*\in \mathcal{P}_{n+1}$ as the corresponding best approximations of $x^k$.

  (a) It is straight forward. (b) Due to the fact that $\|x\|_{\infty}=1$ on the interval $[-1,1]$, we have
  \begin{eqnarray*}
    \mathbf{E}_n(x^k) &=& \|x^k-p_n^*\|_{\infty} = \|x\|_{\infty} \, \|x^k-p_n^*\|_{\infty} \\
    &\geq& \|x(x^k-p_n^*)\|_{\infty} \geq  \mathbf{E}_{n+1}(x^{k+1}).
  \end{eqnarray*}

  (c) Since the interval we consider is symmetric, by substituting $x$ by $-x$, we have $\|x^k-p_{n+1}^*(x)\|_\infty = \|(-x)^k-p_{n+1}^*(-x)\|_\infty$. Additionally, since $x^k=(-x)^k$ when $k$ is even, the above equality becomes $\|x^k-p_{n+1}^*(x)\|_\infty = \|x^k-p_{n+1}^*(-x)\|_\infty$. On account of the uniqueness of $p_{n+1}^*$, it follows that $p_{n+1}^*(x)=p_{n+1}^*(-x)$, for all $x\in[-1,1]$, i.e., $p_{n+1}^*(x)$ is even. Therefore $\mathbf{E}_{n+1}(x^k)=\mathbf{E}_n(x^k)$.

  (d) If $k$ is odd and $n$ is even, according to (b) and (c), we have $\mathbf{E}_n(x^k)\geq \mathbf{E}_{n+1}(x^{k+1})=\mathbf{E}_n(x^{k+1})$. If $k$ is even and $n$ is odd, applying (b) and (c) gives that $\mathbf{E}_n(x^k) = \mathbf{E}_{n-1}(x^{k})\geq \mathbf{E}_n(x^{k+1})$.
\end{proof}

When $k$ is odd and $n$ is even, the next lemma shows a monotonicity relation for the parameter $P_{k,n}$ defined in (\ref{eq:p_k,n}).

\begin{lemma}\label{thm:montonicityP_k,n}
  Let $k$ and $n$ be positive integers with $k\geq n$. If in addition $k+n$ is odd, then the parameter $P_{k,n}$ satisfies
  \[
  P_{k,n} \leq P_{k+2,n}.
  \]
\end{lemma}

\begin{proof}
The parameter $P_{k,n}$ in \eqref{eq:p_k,n} can be interpreted  probabilistically. To this end, let $N_H$ and $N_T$ be the numbers of heads and tails occurring in $k$ coin tosses, and
let $N_{k,n}$ be the numbers of possible cases for which $N_H-N_T>n$. Then
\begin{equation*}
  N_{k,n}=\sum_{j > \frac{n+k}{2}}^{k} \binom{k}{j},
\end{equation*}
and consequently $N_{k,n}=2^{k-1}P_{k,n}$ holds.
Considering the first two out of $k+2$ tosses separately, we have that
\[
N_{k+2,n} = N_{k,n-2} + 2 N_{k,n} + N_{k,n+2},
\]
where the three terms in the right hand side corresponds to the cases that the first two tosses are \{(Head, Head); (Head, Tail) or (Tail, Head); (Tail, Tail)\}.
Since we have $N_{k,n}=2^{k-1} P_{k,n}$ it follows that
$
P_{k+2,n} = \frac{1}{4}\big(P_{k,n-2} + 2 P_{k,n} + P_{k,n+2}\big),
$
which implies that the assertion is true if $P_{k,n-2} + P_{k,n+2} \geq 2 P_{k,n}$. This inequality  now follows from the fact that $\binom{k}{\frac{k+n-1}{2}} \geq \binom{k}{\frac{k+n+1}{2}}$, if $n\geq 2$.
\end{proof}

The last result needed in order to prove Lemma~\ref{thm:P_k,nMontonicity} is the following one.

\begin{lemma}[{\cite[Thm. 3]{newman1976approximation}}]\label{thm:BoundE(x^k)}
  For fixed $k > n$,  the best approximation of $x^k$ on the interval $[-1,1]$ satisfies
    $\frac{1}{4e}P_{k,n}  \leq \mathbf{E}_n(x^k) \leq P_{k,n}.$
\end{lemma}


\begin{proof}[Proof of Lemma~\ref{thm:P_k,nMontonicity}]
  If $k=2M$ or $k=2M-1$, by Lemma~\ref{thm:BoundE(x^k)} and Lemma~\ref{thm:PropertiesEn_x^k}, it follows that $P_{2M-1,2n} \geq \mathbf{E}_{2n}(x^{2M-1}) \geq \mathbf{E}_{2n}(x^{2M})$.

  If $k < 2M-1$ and $k$ is even, $P_{2M-1,2n} \geq P_{2M-3,2n} \geq \cdots\geq P_{k-1,2n} \geq \mathbf{E}_{2n}(x^{k-1}) \geq \mathbf{E}_{2n}(x^{k})$. This equation is obtained by sequentially applying Lemma~\ref{thm:montonicityP_k,n}, Lemma~\ref{thm:BoundE(x^k)} and Lemma~\ref{thm:PropertiesEn_x^k}.

  If $k<2M-1$ and $k$ is odd, by a similar technique, we have $P_{2M-1,2n} \geq P_{2M-3,2n} \geq \cdots\geq P_{k,2n} \geq \mathbf{E}_{2n}(x^{k})$.
\end{proof}

\subsection{Proof of  Lemma~\ref{thm:Convergence_MomentSys_m}}\label{sec:Apx_Pf_Convergence_dynamics}
 In order to prove Lemma~\ref{thm:Convergence_MomentSys_m}, we start with a lemma that gives the error of the best $L_\infty$-approximation for a $k$-order continuously differentiable function on the interval $[-1,1]$.
 Again note that in the entire Appendix~\ref{sec:Apx_Pf_Convergence_dynamics}, for ease of notation, by $\| \cdot \|_{\infty}$ we denote $\| \cdot \|_{\Linf{[-1,1]} }$.

\begin{lemma}[cf. Thm. 16.5 and (16.50) in \cite{powell1981approximation}]\label{thm:BoundE(f)}
 Let a function $f\in \mathbf C^k[-1,1]$, then for $n \geq k$, we have
  \begin{equation*}
    \mathbf{E}_n(f) \leq \big(\frac{\pi}{2}\big)^k \frac{(n-k)!}{n!} \|f^{(k)}\|_\infty.
  \end{equation*}
\end{lemma}

This lemma shows that $\lim_{n\to \infty}\mathbf{E}_n(f)=0$ with a convergence rate $\frac{1}{n^k}$. Next, we investigate some limit properties of the parameter $P_{k,n}$ defined in \eqref{eq:p_k,n}. To this end, the following lemma about the convergence of a specific sequence is needed.
\begin{lemma}\label{thm:convergence_an}
  Let $\ell$ be a fixed nonnegative integer. For any positive constant   $c < e^{-1}$, the sequence
    \[
        a_n=c^n\, \frac{n^{n+\ell}}{n!} \to 0,
    \]
  as $n \to \infty$.
\end{lemma}
\begin{proof}
By using Stirling's formula $n!=\sqrt{2\pi} n^{(n+\frac{1}{2})} e^{-n}(1+\epsilon_n)$, where $\epsilon_n \to 0$ as $n \to \infty$ \cite{robbins1955remark}, the assertion follows.
\end{proof}
%

\begin{prop}\label{thm:limit_E2m,4m-1}
  Let $\ell \in \mathbb{Z}^+$ be a positive integer, then we have
  \[
    \lim_{M\to \infty} M^\ell\, P_{4M-1,2M} =0,
  \]
  which implies that as $M \to \infty$, $\mathbf{E}_{2M}(x^{4M-1})\to 0$ faster than the reciprocal of any polynomial.
\end{prop}
\begin{proof}
  By the definition of $P_{k,n}$, we have
  \begingroup
  \allowdisplaybreaks
  \begin{align*}
    M^\ell P_{4M-1,2M}\!=&\; M^\ell\frac{1}{2^{4M-2}} \sum_{j \geq 3M}^{4M-1} \binom{4M-1}{j} \\
     \leq&\;  M^\ell\,\frac{1}{2^{4M-2}}\, M \, \binom{4M-1}{3M}\\
     =&\;M^{\ell+1} \, \frac{1}{4^{2M-1}}  \frac{1}{(M-1)!} \prod_{j=1}^{M-1}[4M-j]\\
     =&\;M^{\ell+2}\, \frac{1}{M!} \frac{1}{4^{2M-1}} \prod_{j=1}^{M-1}[4M-j]\\
     =&\;M^{\ell+2}\, \frac{1}{M!} \frac{1}{4^{2M-1}} (4M)^{M-1} \!\prod_{j=1}^{M-1}\Big[1-\frac{j}{4M}\Big]\\
     =&\;  M^{\ell+2}\,\Bigg(\frac{ \big(\frac{1}{4}\big) ^M \, M^{M-1}}{M!} \Bigg) \prod_{j=1}^{M-1}\Big[1-\frac{j}{4M}\Big] \\
      \leq &\; \frac{ \big(\frac{1}{4}\big) ^M \, M^{M + \ell + 1}}{M!}.
  \end{align*}
  \endgroup
  Since $1/4 < e^{-1}$ , by Lemma~\ref{thm:convergence_an}, we obtain the assertion.
\end{proof}

Now we are ready to prove Lemma~\ref{thm:Convergence_MomentSys_m}.
\begin{proof}[Proof of Lemma~\ref{thm:Convergence_MomentSys_m}]
 In order to prove the lemma, it is sufficient to prove that
 for any $0<\epsilon <1$, there is an integer $S>0$, such that for any $M \geq S$
  \begin{equation}\label{eq:thm:Convergence_MomentSys_m}
    M  \Big\| \frac{\partial \phi_k}{\partial x} f(x) - p_k^{*}(x) \Big\|_\infty \leq \epsilon,
  \end{equation}
  for each $k=1, 2, \dots, 4M$, where $p_k^*(x) \in \mathcal{P}_{4M}$ is the best $L_\infty$ approximation of the function $(\partial_x \phi_k) f(x)$. Since $f\in \mathbf C^3[-1,1]$, there exist real numbers $B$ and $\bar B$, such that $\|f\|_\infty \leq B$ and $\|f^{(3)}\|_\infty \leq \bar B$. Now, fix an $\varepsilon \in (0,1)$.

  By Lemma~\ref{thm:BoundE(f)}, for $M>1$, the approximation error of $f$ by $\mathcal{P}_{2M}$ satisfies
  \begin{equation*}
    \mathbf{E}_{2M}(f)\leq \Big(\frac{1}{2} \pi \Big)^3 \frac{1}{2M(2M-1)(2M-2)} \, \bar B,
  \end{equation*}
  which implies that $\lim_{M \to \infty} M^2 \, \mathbf{E}_{2M}(f)=0$. Hence, there exits an integer $S_1$ such that $\forall M \geq S_1$, there exists $q_1^*(x) \in \mathcal{P}_{2M}$ such that
  \begin{equation}\label{eq:Proof_DensityThm_1}
    M^2 \big\|f(x)-q_1^*(x)\big\|_\infty \leq \epsilon.
  \end{equation}

  According to Proposition~\ref{thm:limit_E2m,4m-1}, there exists an integer $S_2$ such that $\forall M \geq S_2$, $M^2 \, \mathbf{E}_{2M}(x^{4M-1}) \leq \epsilon$, i.e., there exists a polynomial $q_2^* \in \mathcal{P}_{2M}$ such that
  \begin{equation}\label{eq:Proof_DensityThm_2}
    M \big \|4M x^{4M-1}-q_2^*(x)\big\|_\infty \leq 4\epsilon.
  \end{equation}

  Then we first prove \eqref{eq:thm:Convergence_MomentSys_m} for the case $k=4M$. If $k=4M$, for $M\geq \max\{S_1,S_2\}$, since $q_1^*(x)\, q_2^*(x) \in \mathcal{P}_{4M}$, we have
  \begin{align}
    & M\, \Big\| \frac{\partial \phi_k}{\partial x} f(x) - p_k^{*}(x) \Big\|_\infty \!
    \leq M\,  \Big\| \frac{\partial \phi_k}{\partial x} f(x) - q_1^*(x)\, q_2^*(x) \Big\|_\infty\ \nonumber \\
    &= M \Big\| \frac{\partial \phi_k}{\partial x} f -q_2^*\, f + q_2^*\, f- q_1^*\, q_2^* \Big\|_\infty\ \nonumber \\
    &\leq  M \|f\|_\infty \big\| 4M x^{4M-1} - q_2^* \big\|_\infty  + M\|q_2^*\|_\infty  \|f-q_1^*\|_\infty\ \nonumber  \\
     &\leq  4B\epsilon + M\, \big\|q_2^*-4Mx^{4M-1}+4Mx^{4M-1}\big\|_\infty\, \|f-q_1^*\|_\infty \nonumber \\
     &\leq 4B\epsilon + \Big(4\epsilon +4M^2 \big\|x^{4M-1}\big\|_\infty \Big) \,   \|f-q_1^*\|_\infty \nonumber \\
       &\leq  (4B+ 8) \epsilon, \label{eq:proof_DensityThm_4}
  \end{align}
  where the second inequity comes from the triangle inequality and the fact that for any $L_\infty$-functions $g,h$, $\| gh \|_\infty \leq \|g\|_\infty \|h\|_\infty$, the third and the fourth inequities come from \eqref{eq:Proof_DensityThm_2},
  and the last inequality follows from  \eqref{eq:Proof_DensityThm_1} and $x\in [-1,1]$.
  By rescaling the arbitrarily small $\epsilon$, the argument for $k=4M$ is proved.

  Next, we consider the situation for any $k< 4M$.
Due to the existence and uniqueness of the $L_\infty$-approximation \cite[Thm.~7.6]{powell1981approximation}, we know that there exists a $q_2^k \in \mathcal{P}_{2M}$ such that $\mathbf{E}_{2M}(x^{k-1})= \|x^{k-1}-\frac{1}{k}q_2^k\|_\infty$. Therefore

  \begin{eqnarray}
    M\, \|k\, x^{k-1} -q_2^k\|_\infty &<&  4M^2 \|x^{k-1}-\frac{1}{k}q_2^k\|_\infty \nonumber \\
     &=& 4M^2\mathbf{E}_{2M}(x^{k-1}) \nonumber \\
     &\leq& 4M^2 P_{4M-1,2M} \leq 4\epsilon, \label{eq:proof_DensityThm_3}
  \end{eqnarray}
  where the first inequality comes from $k< 4M$, the second inequality is due to Lemma~\ref{thm:P_k,nMontonicity}, and the last inequality is due to Proposition~\ref{thm:limit_E2m,4m-1} for large enough $M$. Then we employ a similar procedure as that in (\ref{eq:proof_DensityThm_4}) and use equation   (\ref{eq:proof_DensityThm_3}), by which \eqref{eq:thm:Convergence_MomentSys_m} can be obtained  for case $k<4M$.
\end{proof}
Actually, by imposing more smoothness on $f$, the convergence rate in Lemma~\ref{thm:Convergence_MomentSys_m} can be faster than any polynomial order, which we formalize in the following corollary.

\begin{cor}\label{thm:Convergence_MomentSys_m^l}
 If function $f$ in Lemma~\ref{thm:Convergence_MomentSys_m} is $(\ell+2)$-continuously differentiable, i.e., $f\in \mathbf C^{\ell+2}[-1,1]$, for a fixed $\ell\in \mathbb{Z}^+$, then the lemma holds with the equation (\ref{eq:thm:Convergence_MomentSys_m1}) replaced by
    \begin{equation*}
     \max_{k=1, \dots, 4M} { M^\ell \mathbf{E}_{4M}\Big(f(x)\partial_x \phi_k(x)  \Big)} \to 0.
  \end{equation*}
\end{cor}

Now by the means of Lemma~\ref{thm:Convergence_MomentSys_m}, Theorem~\ref{thm:Convergence_MomentSys_2Norm} can be easily proved. Moreover Corollary~\ref{thm:Convergence_MomentSys_2Norm_m^l} follows directly from Corollary~\ref{thm:Convergence_MomentSys_m^l}.

\subsection{Proof of Lemma~\ref{lem:limit_E2m+1,2m}}
\label{sec:Apx_Pf_limit_E2m+1,2m}
  This argument can be easily verified by rescaling the coordinates. We let $x=\frac{b-a}{2}y+\frac{b+a}{2}$, where $y\in [-1,1]$. This means that $x^{n+1} = (\frac{b-a}{2})^{n+1}y^{n+1}+p_n(y)$, where $p_n(y)$ is a polynomial of degree $n$. Clearly such a polynomial can be exactly approximated by monomials of degree $\leq n$, and thus
\begin{align*}
  \mathbf{E}_n^{[a,b]}(x^{n+1}) &= \mathbf{E}_n \left( \left( \frac{b-a}{2} \right)^{n+1}y^{n+1}+p_n(y) \right) \\
  &=  \left( \frac{b-a}{2}\right)^{n+1}  \mathbf{E}_n(y^{n+1}),
\end{align*}
where $y \in [-1,1]$. Then by Lemma~\ref{thm:BoundE(x^k)},
\begin{align*}
    \mathbf{E}_n^{[a,b]}(x^{n+1})
    &\geq \frac{1}{4e} \Big(\frac{b-a}{2}\Big)^{n+1} \, P_{n+1,n}
    \geq \frac{1}{2e} \Big(\frac{b-a}{4}\Big)^{n+1}.
\end{align*}
  Since $b-a>4$, the assertion follows. \hfill$\square$





\bibliographystyle{plain}
\bibliography{bib_johan_abbrv}

\vspace{-8mm}
\begin{IEEEbiography}
    [{\includegraphics[width=1in,height=1in,clip,keepaspectratio]{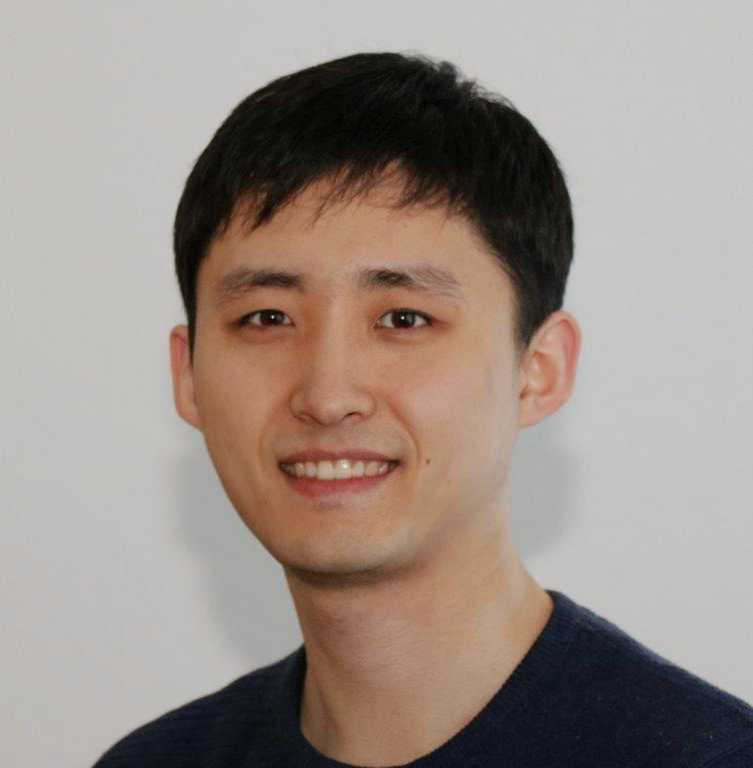}}]{Silun Zhang} (S'16--M'20) received his B.Eng. and M.Sc. degrees in Automation from Harbin Institute of Technology, China, in 2011 and 2013 respectively, and the PhD degree in Optimization and Systems Theory from Department of Mathematics, KTH Royal Institute of Technology, Sweden, in 2019.

He is currently a postdoctoral researcher with the Laboratory for Information and Decision Systems (LIDS), MIT, USA. His main research interests include nonlinear control, networked systems, rigid-body attitude control, and modeling large-scale systems.

\end{IEEEbiography}
\vspace{-8mm}

\begin{IEEEbiography}
    [{\includegraphics[width=1in,height=1.25in,clip,keepaspectratio]{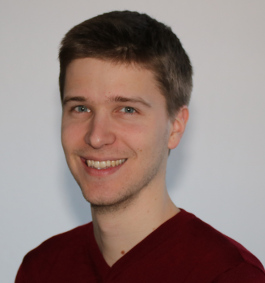}}]{Axel Ringh} (S'15--M'20) received a M.Sc. degree in Engineering Physics in 2014, and a Ph.D. degree in Applied and Computational Mathematics in 2019, both from KTH Royal Institute of Technology, Stockholm, Sweden. He is currently a Wallenberg postdoctoral researcher with the Department of Electronic and Computer Engineering, The Hong Kong University of Science and Technology, Hong Kong, China. His current research are within the areas of control theory and signal processing, and include analytic interpolation problems, moment problems, optimal mass transport, and methods for convex optimization.
\end{IEEEbiography}

\vspace{-8mm}
\begin{IEEEbiography}
    [{\includegraphics[width=1in,height=1.25in,clip,keepaspectratio]{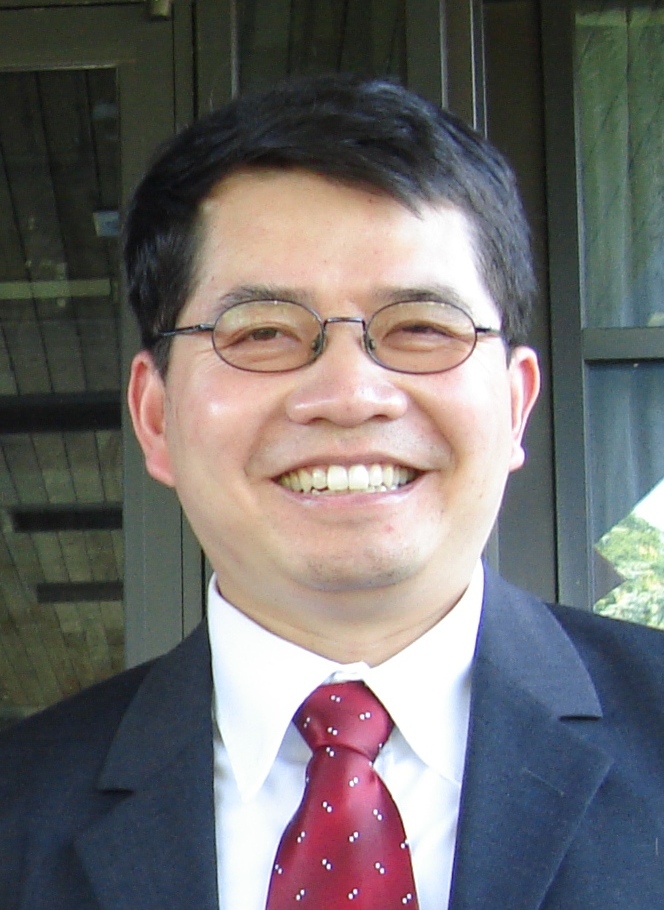}}]{Xiaoming Hu}
received the B.S. degree from the University of Science and Technology of China in 1983, and the M.S. and Ph.D. degrees from the Arizona State University in 1986 and 1989 respectively. He served as a research assistant at the Institute of Automation, the Chinese Academy of Sciences, from 1983 to 1984. From 1989 to 1990 he was a Gustafsson Postdoctoral Fellow at the Royal Institute of Technology, Stockholm, where he is currently a professor of Optimization and Systems Theory. His main research interests are in nonlinear control systems, nonlinear observer design, sensing and active perception, motion planning, control of multi-agent systems, and mobile manipulation.
\end{IEEEbiography}

\vspace{-8mm}

\begin{IEEEbiography}
    [{\includegraphics[width=1in,height=1.25in,clip,keepaspectratio]{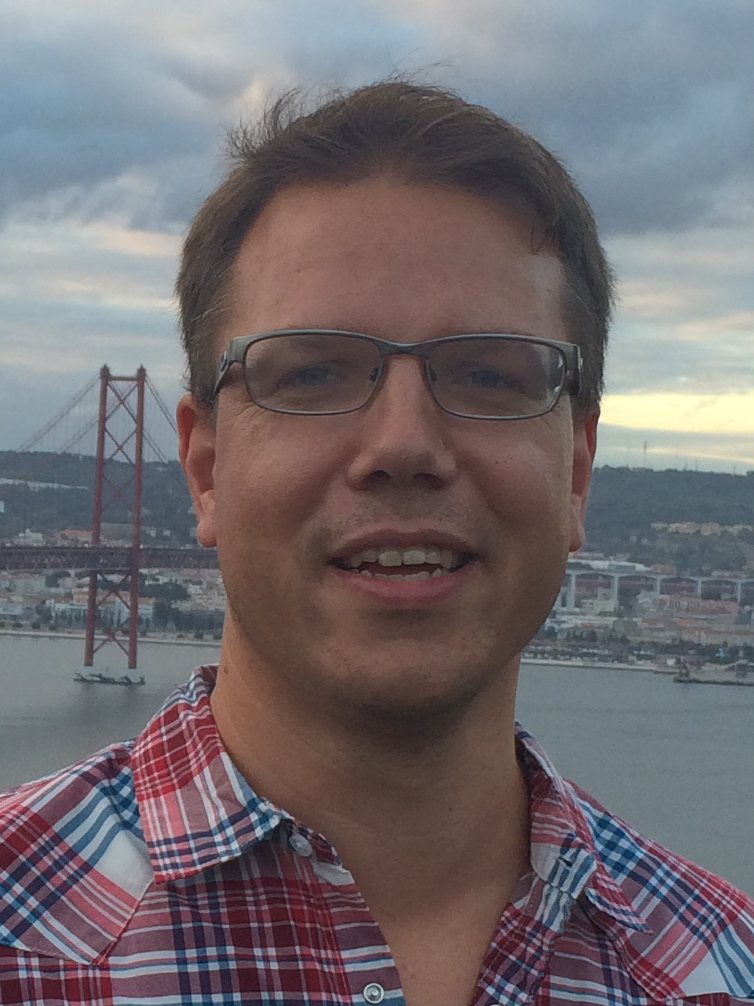}}]{Johan Karlsson} (S'06--M'09--SM'18)
was born in Stockholm, Sweden, 1979. He received an M.Sc. degree in Engineering Physics from Royal Institute of Technology (KTH) in 2003 and a Ph.D. in Optimization and Systems Theory from KTH in 2008. From 2009 to 2011, he was with Sirius International, Stockholm; and from 2011 to 2013 he was working as a postdoctoral associate at the Department of Computer and Electrical Engineering, University of Florida. From 2013 he joined the Department of Mathematics, KTH, as an assistant professor and since 2017 he is working as an associate professor. He has been the main organizer of several workshops, in particular for establishing collaborations between the academia and the industry. His current research interests include inverse problems, methods for large scale optimization, and model reduction, for applications in remote sensing, signal processing, and control theory.
%
\end{IEEEbiography}


\end{document}